\documentclass[12pt, a4paper]{amsart}
\usepackage{amscd, amssymb, pb-diagram}
\usepackage{amsthm}
\def\newspan{\operatorname{span}}
\def\range{\operatorname{range}}
\def\supp{\operatorname{supp}}
\def\ker{\operatorname{ker}}

\def\id{\operatorname{id}}

\def\sce{\operatorname{sce}}
\def\fin{\operatorname{fin}}
\def\rg{\operatorname{rg}}

\def\sg{\operatorname{sg}}
\def\clsp{\overline{\operatorname{span}}}
\def\Orb{\operatorname{Orb}}
\def\Per{\operatorname{Per}}
\def\BV{\operatorname{BV}}

\def\N{\mathbb{N}}
\def\Z{\mathbb{Z}}
\def\T{\mathbb{T}}

\def\TT{\mathcal{T}}
\def\LL{\mathcal{L}}
\def\OO{\mathcal{O}}
\def\KK{\mathcal{K}}

\def\HH{\mathcal{H}}

\def\XX{\mathcal{X}}

\def\MM{\mathcal{M}}

\def\FF{\mathcal{F}}

\newtheorem{thm}{Theorem}[section]
\newtheorem{cor}[thm]{Corollary}
\newtheorem{lemma}[thm]{Lemma}
\newtheorem{prop}[thm]{Proposition}

\theoremstyle{definition}

\theoremstyle{remark}
\newtheorem{remark}[thm]{Remark}
\newtheorem{example}[thm]{Example}

\numberwithin{equation}{section}
\begin{document}
\title[Exel's crossed product]{\boldmath{Exel's crossed product for non-unital $C^*$-algebras}}
\author{Nathan Brownlowe}

\author{Iain Raeburn}

\author{Sean T. Vittadello}

\address{Nathan Brownlowe, Iain Raeburn, Sean T. Vittadello, School of Mathematical and Applied Statistics\\
University of Wollongong\\
NSW 2522\\
Australia}
\email{nathanb@uow.edu.au, raeburn@uow.edu.au, seanv@uow.edu.au}

\begin{abstract}
We consider a family of dynamical systems $(A,\alpha,L)$ in which $\alpha$ is an endomorphism of a $C^*$-algebra $A$ and $L$ is a transfer operator for $\alpha$. We extend Exel's construction of a crossed product to cover non-unital algebras $A$, and show that the $C^*$-algebra of a locally finite graph can be realised as one of these crossed products. When $A$ is commutative, we find criteria for the simplicity of the crossed product, and analyse the  ideal structure of the crossed product. 
\end{abstract}
\thanks{This research has been supported by the Australian Research Council}
\maketitle

\section{Introduction}\label{Introduction}

Crossed products of $C^*$-algebras by endomorphisms were first used to describe the relationship between the Cuntz algebras $\OO_n$ and their UHF cores \cite{c, pas}; the original constructions were spatial, and Stacey later described an appropriate universal construction \cite{s}. Various generalisations to semigroups of endomorphisms have been proposed \cite{mold, n, m}, and these crossed products have been used to study Toeplitz algebras and Hecke algebras \cite{alnr,lr2,lr}. The endomorphisms in these applications have all been non-unital \emph{corner endomorphisms}, which shift the algebra onto a full corner of itself. 

In \cite{e1}, Exel observed that these notions of crossed product do not work well for the endomorphisms coming from classical dynamical systems in which the dynamics is irreversible, and proposed an alternative construction. The crucial extra ingredient in Exel's construction is a \emph{transfer operator}: a positive linear map which is, loosely speaking, a left inverse for the endomorphism.  One of his main motivations was to find a version of the crossed-product construction which realised the Cuntz-Krieger algebras as crossed products by a single endomorphism. His answer to this problem is quite different from Cuntz's description of $\OO_n$: Exel realises a Cuntz-Krieger algebra as a crossed product of the diagonal subalgebra, which is a maximal commutative subalgebra, and is much smaller than the UHF core in $\OO_n$. 

In most of these examples and applications, the underlying $C^*$-algebras have identities, even though many of the endomorphisms are not unital. For example, in Exel's description of the Cuntz-Krieger algebras, the underlying algebra is the (unital) algebra of continuous functions on a compact space of infinite words. Recently there have been many interesting generalisations of Cuntz-Krieger algebras, such as the graph algebras discussed in \cite{r}, where the infinite-path space is locally compact rather than compact.  Our goal here is to extend Exel's construction to cover endomorphisms of non-unital algebras, with a view to realising some substantial family of graph algebras as Exel crossed products.

Our extension of Exel's construction follows the original as closely as possible: there are technical issues involving nondegeneracy of representations and homomorphisms, but otherwise things go quite smoothly. Our main technical tools are a realisation of the crossed product as a relative Cuntz-Pimsner algebra, generalising the one for unital algebras found and used by the first two authors in \cite{br}, and a closely related realisation as a topological-graph algebra, which allows us to apply the deep results of Katsura on simplicity and ideal structure \cite{k,k2}. We succeed in realising the $C^*$-algebras of locally finite graphs without sources as Exel crossed products, and we analyse the ideal structure of Exel crossed products arising from (non-compact) irreversible dynamical systems. The limitations of our method (for example, as to what kinds of graphs we can handle) are in many ways as interesting as the results we have obtained, and at the end we make some speculative comments on what we have learned from our  investigations.

\smallskip

We begin in Section~\ref{Exelsys} by describing the \emph{Exel systems} which we study. Each system consists of an endomorphism $\alpha$ of a $C^*$-algebra $A$ and a transfer operator $L:A\to A$. For technical reasons, we have chosen to assume that the endomorphisms and transfer operators have strictly continuous extensions to the multiplier algebra; similar \emph{extendibility} hypotheses have appeared in the work of Adji \cite{a} and Larsen \cite{l}. These properties are enjoyed by the endomorphisms $\alpha:f\mapsto f\circ\tau$ of $C_0(T)$ associated to proper local homeomorphisms $\tau:T\to T$; we refer to such a pair $(T,\tau)$ as a \emph{classical system}. In our motivating example, $\tau$ is the shift on the infinite path space of a locally finite graph.

In Sections~\ref{The Toeplitz crossed product} and~\ref{The crossed product}, we describe the crossed products of Exel systems $(A,\alpha,L)$. As in \cite{e1}, there are two algebras of interest: the \emph{Toeplitz crossed product} $\TT(A,\alpha,L)$, and the \emph{crossed product} $A{\rtimes}_{\alpha,L}\N$, which is a quotient of $\TT(A,\alpha,L)$. Following \cite{br}, we identify $\TT(A,\alpha,L)$ as the Toeplitz algebra of a particular Hilbert bimodule $M_L$ built from $(A,\alpha,L)$ (Proposition~\ref{the Toeplitz algebra}), and  $A{\rtimes}_{\alpha,L}\N$ as a relative Cuntz-Pimsner algebra $\OO(K_\alpha,M_L)$ (Theorem~\ref{our redundancies}). For Exel systems $(C_0(T),\alpha,L)$ arising from classical systems, the ideal $K_\alpha$ is all of $A$, and $C_0(T) {\rtimes}_{\alpha,L}\N$ is the Cuntz-Pimsner algebra $\OO(M_L)$. 

In Section~\ref{realisegralg}, we achieve one of our goals by proving that the $C^*$-algebra of a locally finite graph with no sources can be realised as the Exel crossed product of the classical system involving the shift of the (locally compact) space of infinite paths. (Exel and Royer \cite{er} have described a different extension of the theory in \cite{e1} which covers the Exel-Laca algebras using a (unital) algebra of functions on a compact space.)

In Section~\ref{The system (C_0(X),alpha,L)}, we give criteria for the simplicity of crossed products associated to classical systems.  Our main tool is the work of Katsura \cite{k,k2}, which applies because we can realise the Cuntz-Pimsner algebra $\OO(M_L)=C_0(T)\rtimes_{\alpha,L}\N$ as the $C^*$-algebra of a topological graph. We then check that these criteria are compatible with the known criteria for graph algebras. In Sections~\ref{secgauge} and~\ref{secprim} we use the same technique to determine the gauge-invariant ideals and primitive ideals of crossed products of the form $C_0(T)\rtimes_{\alpha,L}\N$. In all these sections, it takes some effort to recast the results in the language of dynamics so we can compare them with those for compact $T$ in \cite{ev}, and more effort to convert them to the usual graph-theoretic descriptions of the ideal structure of graph algebras in \cite{bprs, bhrs, hs, r}, for example.  Reassuringly, though, everything does match up in the end.

\subsection{Background and notation} 

Let $A$ be a $C^*$-algebra. A {\em Hilbert $A$-bimodule} (or {\em correspondence over} $A$) is a right Hilbert $A$-module $M$ together with a left action of $A$ on $M$ which is implemented by a homomorphism $\phi$ of $A$ into the $C^*$-algebra $\LL(M)$ of adjointable operators on $M$: $a\cdot x:=\phi(a)(x)$. A {\em Toeplitz representation} $(\psi,\pi)$ of $M$ in a $C^*$-algebra $B$ consists of a linear map $\psi:M\to B$ and a homomorphism $\pi:A\to B$ such that
\[
\psi(x\cdot a)=\psi(x)\pi(a),\quad\psi(x)^*\psi(y)=\pi({\langle x,y\rangle}_A),\quad\text{and}\quad\psi(a\cdot x)=\pi(a)\psi(x).
\]  
The {\em Toeplitz algebra} of $M$ is the $C^*$-algebra $\TT(M)$ generated by a universal Toeplitz representation $(i_M,i_A)$ (see \cite[Proposition~1.3]{fr}). 

For $x,y\in M$ the operator $\Theta_{x,y}:M\to M$ defined by $\Theta_{x,y}(z):= x\cdot{\langle y,z\rangle}_A$ is adjointable with $\Theta_{x,y}^*=\Theta_{y,x}$. The span $\KK(M):=\overline{\newspan}\{\Theta_{x,y}:x,y\in M\}$ is a closed two-sided ideal in $\LL(M)$ called the {\em algebra of compact operators on} $M$. Thus $J(M):=\phi^{-1}(\KK(M))$ is a closed two-sided ideal in $A$. For every Toeplitz representation $(\psi,\pi)$ of $M$ in $B$ there is a homomorphism ${(\psi,\pi)}^{(1)}:\KK(M)\to B$ satisfying
\[
{(\psi,\pi)}^{(1)}(\Theta_{x,y})=\psi(x)\psi(y)^*\quad\text{for $x,y\in M$.}
\]
If $K$ is an ideal with $K\subset J(M)$, a Toeplitz representation $(\psi,\pi)$ of $M$ is {\em coisometric on} $K$ if
\[
{(\psi,\pi)}^{(1)}(\phi(a))=\pi(a)\quad\text{for $a\in K$,}
\]
and the {\em relative Cuntz-Pimsner algebra} $\OO(K,M)$ is the $C^*$-algebra generated by a universal Toeplitz representation $(k_M,k_A)$ which is coisometric on $K$ (see \cite{ms, fmr}). It is the quotient of $\TT(M)$ by the ideal generated by
\[
\{{(i_M,i_A)}^{(1)}(\phi(a))-i_A(a)):a\in K\},
\]
and if $q:\TT(M)\to\OO(K,M)$ is the quotient map, then $(k_M,k_A):=(q\circ i_M,q\circ i_A)$. We have $\OO(\{0\},M)=\TT(M)$, and $\OO(J(M),M)$ is Pimsner's version of the Cuntz-Pimsner algebra \cite{p,fmr}. With ${(\ker\phi)}^{\perp}=\{a\in A:ab=0\text{ for all $b\in\ker\phi$}\}$, we recover Katsura's version of the Cuntz-Pimsner algebra as $\OO(J(M)\cap {(\ker\phi)}^{\perp},M)$ \cite{kat}.  In our bimodules the homomorphism $\phi$ is always injective, and Pimsner's and Katsura's Cuntz-Pimsner algebras are the same algebra $\OO(M)$.

\section{Exel systems}\label{Exelsys}

Suppose $A$ is a $C^*$-algebra and $\alpha$ is an endomorphism of $A$. We assume throughout that $\alpha$ is \emph{extendible}: there is a strictly continuous endomorphism $\overline{\alpha}$ of $M(A)$ such that $\overline{\alpha}|_A=\alpha$. This is equivalent to assuming that there is an approximate identity ${(u_{\lambda})}_{\lambda\in\Lambda}$ for $A$ and a projection $p_{\alpha}\in M(A)$ such that $\alpha(u_{\lambda})\longrightarrow p_{\alpha}$ strictly in $M(A)$. In this paper, a \emph{transfer operator} $L$ for $(A,\alpha)$ is a bounded positive linear map $L:A\rightarrow A$ which extends to a bounded positive linear map $\overline{L}:M(A)\rightarrow M(A)$ such that $L(\alpha(a)m)=a\overline{L}(m)$ for $a\in A$ and $m\in M(A)$. We call the triple $(A,\alpha,L)$ an \emph{Exel system}.

\begin{remark}\label{props of to's}
Since positive linear maps are adjoint-preserving, we also have $L(m\alpha(a))=
\overline{L}(m)a$. Such transfer operators $L$ are automatically strictly continuous.
\end{remark}

\subsection{Exel systems arising from classical systems}\label{subsectcomm}
In the main examples of interest to us (and in \cite{e1}, \cite{e2} and \cite{ev}), the $C^*$-algebra $A$ is commutative. A \emph{classical system} consists of a locally compact space $T$ and a local homeomorphism $\tau:T\to T$ which is proper in the sense that inverse images of compact sets are compact. Properness implies that $\alpha:f\mapsto f\circ\tau$ maps $C_0(T)$ into $C_0(T)$, and the endomorphism $\alpha$ is nondegenerate, hence extendible with $\overline{\alpha}(1)=1$. As in \cite{e1} and \cite{e2}, the transfer operator $L$ is defined by averaging over the inverse images of points. It is not immediately obvious that this process maps $C_0(T)$ to itself:

\begin{lemma}\label{Lfcts}
Suppose that $\tau:T\to T$ is a proper local homeomorphism. Then the function $\delta:T\to \N$ defined by $\delta(t)=|\tau^{-1}(t)|$ is locally constant, and for every $f\in C_0(T)$ the function $L(f)$ defined by
\begin{equation}\label{defL}
L(f)(t)=\frac{1}{|\tau^{-1}(t)|}\sum_{\tau(s)=t}f(s)
\end{equation}
belongs to $C_0(T)$.
\end{lemma}

\begin{proof}
We fix $t\in T$ and a compact neighbourhood $N$ of $t$. The inverse image $\tau^{-1}(t)$ is a compact set, and it cannot have a cluster point because $\tau$ is a local homeomorphism, so it must be finite. We list it as $\tau^{-1}(t)=\{s_i:1\leq i\leq m\}$. Next choose disjoint open sets $U_i\subset \tau^{-1}(N)$ such that $s_i\in U_i$ and $s|_{U_i}$ is a local homeomorphism onto an open neighbourhood of $t$. The set $K:=\tau^{-1}(N)\setminus\big(\bigcup_{i}U_i\big)$ is compact, and $t$ does not belong to $\tau(K)$, so there is a neighbourhood $V$ of $t$ which misses $\tau(K)$. Then $W:=\bigcap_i(V\cap\tau(U_i))$ is an open neighbourhood of $t$, and every point of $W$ has exactly $m$ preimages, one in each $U_i$. So $\delta$ is constant on $W$, and
$L(f)|_W=\frac{1}{m}\sum_{i=1}^mf\circ(s|_{U_i})^{-1}|_{W}$ is continuous at $t$.

Finally, note that if $|f|<\epsilon$ outside a compact set $K$, then $|L(f)|<\epsilon$ outside the compact set $\tau(K)$. 
\end{proof}

Calculations show that the map $L:C_0(T)\to C_0(T)$ defined in Lemma~\ref{Lfcts} is positive, norm-decreasing and satisfies $L(\alpha(f)g)=fL(g)$. Equation \eqref{defL} also defines a map $\overline{L}$ on $C_b(T)=M(C_0(T))$ with the required properties, and hence $L$ is a transfer operator for $\alpha$. Thus $(C_0(T),\alpha,L)$ is an Exel system.

\begin{remark}\label{normalising factor}
The normalising factor of $|\tau^{-1}(t)|^{-1}$ in \eqref{defL} is not required for the key identity $L(\alpha(f)g)=fL(g)$ ---  we could multiply $L$ by any bounded continuous function without changing this equation. Indeed, in \cite{ev} no normalising factor is used. However, there the space $T$ is compact, so the  function $t\mapsto |\tau^{-1}(t)|$ is bounded, and the unnormalised transfer operator is still a bounded linear map on $C(T)$. When $T$ is locally compact, $t\mapsto |\tau^{-1}(t)|$ need not be bounded, and then we have to include the normalising factor to ensure that \eqref{defL} defines a bounded operator on $C_0(T)$.
\end{remark}

\subsection{Systems arising from directed graphs}\label{intrographex}

We assume throughout this paper that $E=(E^0,E^1,r,s)$ is a locally finite directed graph with no sources, and in \S\ref{Conclusions} we discuss the changes that would need to be made to accommodate more general graphs. We think of elements of $E^0$ as vertices, elements of $E^1$ as edges, and $r,s:E^1\to E^0$ as determining the range and source of edges. Saying that $E$ has no sources means that $r^{-1}(v)$ is nonempty for every vertex $v\in E^0$. Local-finiteness means that $E$ is both \emph{row-finite} ($r^{-1}(v)$ is finite for every $v$) and \emph{column-finite} ($s^{-1}(v)$ is finite for every $v$). 

We use the conventions of \cite{r} for graphs and their $C^*$-algebras. Thus $C^*(E)$ is the $C^*$-algebra generated by a universal Cuntz-Krieger $E$-family consisting of partial isometries $\{s_e:e\in E^1\}$ and mutually orthogonal projections $\{p_v:v\in E^0\}$ such that $s_e^*s_e=p_{s(e)}$ and $p_v=\sum_{r(e)=v}s_es_e^*$. We write $E^*$ for the set of finite paths $\mu=\mu_1\mu_2\cdots\mu_n$ satisfying $s(\mu_i)=r(\mu_{i+1})$ for all $i$, and $|\mu|$ for the length $n$ of such a path $\mu$.

The Exel system associated to $E$ arises from a classical system, as in  \S\ref{subsectcomm}. The underlying topological space $E^\infty$ is the set of infinite paths $\xi=\xi_1\xi_2\xi_3\cdots$, which is locally compact in the product topology from $\prod_{n=1}^\infty E^1$ because $E$ is row-finite; this topology has a basis consisting of the compact open sets $Z(\mu):=\{\xi\in E^\infty:\xi_i=\mu_i\text{ for $i\leq |\mu|$}\}$ for $\mu\in E^*$. The map $\sigma$ is the shift on $\sigma:E^\infty\to E^\infty$ defined by $\sigma(\xi_1\xi_2\xi_3\cdots)=\xi_2\xi_3\cdots$; $\sigma$ is a local homeomorphism because it is a homeomorphism of each $Z(e)$ onto $Z(s(e))$, and is proper because the graph is column-finite. 

As in \S\ref{subsectcomm}, the endomorphism $\alpha$ in our Exel system $(C_0(E^\infty),\alpha, L)$ is given by $\alpha:f\mapsto f\circ\sigma$ and the transfer operator $L$ is defined by averaging over the inverse images of points. Since $\sigma^{-1}(\xi)=\{e\xi:s(e)=r(\xi)\}$, we can write $L$ as
\[
L(f)(\xi)=\frac{1}{|s^{-1}(r(\xi))|}\sum_{s(e)=r(\xi)}f(e\xi).
\]
Even for locally finite graphs $E$ the valencies $|s^{-1}(v)|$ may be unbounded, so this is one situation where we need the normalising factor to make $L$ bounded (see Remark~\ref{normalising factor}).

\section{The Toeplitz crossed product}\label{The Toeplitz crossed product}

A \emph{Toeplitz-covariant representation} of an Exel system $(A,\alpha,L)$ in a $C^*$-algebra $B$ consists of a nondegenerate homomorphism $\pi:A\to B$
and an element $V\in M(B)$ such that 
\begin{itemize}\item[]\begin{itemize}
\item[(TC1)] $V\pi(a)=\pi(\alpha(a))V$, and 
\smallskip
\item[(TC2)] $V^*\pi(a)V=\pi(L(a))$.
\end{itemize} 
\end{itemize}
The \emph{Toeplitz crossed product} $\TT(A,\alpha,L)$ is the $C^*$-algebra  generated by a universal Toeplitz-covariant representation $(i,S)$.

Following \cite{e1} and \cite{br}, we next realise $\TT(A,\alpha,L)$ as the Toeplitz algebra of a Hilbert bimodule.  We make $A$ into a right $A$-module $A_L$ in which the right action of $a\in A$ on $m\in A_L$ is given by $m\cdot a=m\alpha(a)$, and define a pairing on $A_L$ by ${\langle m,n\rangle}_L=L(m^*n)$; $A_L$ is then a pre-inner-product module. The completion $M_L$ is a Hilbert $A$-module. We denote the quotient map by $q:A_L\to M_L$. The action of $A$ by left multiplication extends to an action by bounded adjointable operators on $M_L$, giving a homomorphism $\phi:A\to\LL(M_L)$, and $M_L$ becomes a right-Hilbert bimodule. Further details are in \cite[\S2]{br}.  An approximate-identity argument shows that $M_L$ is essential as a left $A$-module: $A\cdot M_L=\{a\cdot m:a\in A,m\in M_L\}$ is dense in $M_L$. ($M_L$ is also essential as a right $A$-module, because every Hilbert module is \cite[Corollary~2.7]{rw}.)

\begin{prop}\label{the Toeplitz algebra}
Suppose  $(A,\alpha,L)$ is an Exel system. There is a linear map ${\psi}_S:M_L\rightarrow\TT(A,\alpha,L)$ such that ${\psi}_S(q(a))=i(a)S$, and $({\psi}_S,i)$ is a Toeplitz representation of $M_L$ in $\TT(A,\alpha,L)$ such that $\psi_S\times i$ is an isomorphism of $\TT(M_L)$ onto $\TT(A,\alpha,L)$.
\end{prop}

This proposition seems to be substantially trickier when $A$ does not have an identity. As in the unital case, there is an issue with nondegeneracy: in a Toeplitz representation $(\psi,\pi)$, the representation $\pi$ does not have to be nondegenerate. But even if we assume nondegeneracy, it is not so easy to move from Toeplitz representations $(\psi,\pi)$ to Toeplitz-covariant representations $(\pi,V)$: in the unital case, we just take $V=\psi(1)$, and we go back by taking $\psi_V(q(a))=\pi(a)V$ (see \cite[Lemma~3.2]{br}). Here we construct $V$ from $(\psi,\pi)$ using a spatial argument.

\begin{lemma}\label{corres. between T-c repns and T repns p1} 
Suppose $(\mu,\tau)$ is a Toeplitz representation of $M_L$ on a Hilbert space $\HH$, and $\tau$ is nondegenerate. Then there is a bounded linear operator $U_{\mu,\tau}$ on $\HH$ such that 
\begin{equation}\label{defU}
U_{\mu,\tau}\Big(\sum_{i=1}^m\tau(a_i)k_i\Big)=\sum_{i=1}^m\mu(q(\alpha(a_i)))k_i\quad\text{for $a_i\in A$ and $k_i\in\HH$,} 
\end{equation}
and the pair $(\tau,U_{\mu,\tau})$ is a Toeplitz-covariant representation on $\HH$. 
\end{lemma}

\begin{proof}
Nondegeneracy ensures that $\tau$ extends to a representation $\overline{\tau}:M(A)\to B(\HH)$, and a calculation using the equation $L(\alpha(a)\alpha(b))=a\overline{L}(1)b$ shows that
\begin{equation}\label{eq for sums}
\Big\|\sum_{i=1}^m\mu(q(\alpha(a_i)))k_i\Big\|^2 \leq\big\|\overline{\tau}(\overline{L}(1))\big\|\;\Big\|\sum_{i=1}^m\tau(a_i)k_i\Big\|^2\leq\|\overline{L}(1)\|\;\Big\|\sum_{i=1}^m\tau(a_i)k_i\Big\|^2.
\end{equation}
If $\sum_{i=1}^m\tau(a_i)k_i=\sum_{i=1}^n\tau(b_i)l_i$, then \eqref{eq for sums} implies that
\[
\Big\|\sum_{i=1}^m\mu(q(\alpha(a_i)))k_i-\sum_{i=1}^n\mu(q(\alpha(b_i)))l_i\Big\|^2 \le \big\|\overline{L}(1)\big\|{\Big\|\sum_{i=1}^m\tau(a_i)k_i-\sum_{i=1}^n\tau(b_i)l_i\Big\|}^2= 0,
\]
and hence  there is a well-defined linear map $U_{\mu,\tau}$ on $\newspan\{\tau(a)h:a\in A,h\in \HH\}$ satisfying \eqref{defU}. Equation~\eqref{eq for sums} implies that $U_{\mu,\tau}$ is norm-decreasing, and hence extends to a bounded linear operator on $\clsp\{\tau(a)h:a\in A,h\in \HH\}$, which is all of $\HH$ by nondegeneracy of $\tau$.

To see that $(\tau,U_{\mu,\tau})$ is Toeplitz-covariant, we let $b\in A$. Then  
\begin{align*}
U_{\mu,\tau}\tau(a)(\tau(b)h) &= \mu(q(\alpha(ab)))h
= \tau(\alpha(a))\mu(q(\alpha(b)))h= \tau(\alpha(a))U_{\mu,\tau}(\tau(b)h),
\end{align*}
and the nondegeneracy of $\tau$ implies that $U_{\mu,\tau}\tau(a)=\tau(\alpha(a))U_{\mu,\tau}$. Next we calculate:
\begin{align*}
\big({U_{\mu,\tau}}^*\tau(a)U_{\mu,\tau}(\tau(b)h)\,|\,\tau(c)k\big) &= \big(\tau(a)U_{\mu,\tau}(\tau(b)h)\,|\,U_{\mu,\tau}(\tau(c)k)\big)\\
&= \big(\tau(a)\mu(q(\alpha(b)))h\,|\,\mu(q(\alpha(c)))k\big)\\
&=\big({\mu(q(\alpha(c)))}^*\mu(q(a\alpha(b)))h\,|\,k\big)\\
&= \big(\tau(L({\alpha(c)}^*a\alpha(b)))h\,|\,k\big)\\
&=\big({\tau(c)}^*\tau(L(a))(\tau(b)h)\,|\,k\big)\\
&= \big(\tau(L(a))(\tau(b)h)\,|\,\tau(c)k\big),
\end{align*}
which gives ${U_{\mu,\tau}}^*\tau(a)U_{\mu,\tau}=\tau(L(a))$.
\end{proof} 

\begin{lemma}\label{defpsiV}
If $(\pi,V)$ is a Toeplitz representation of $(A,\alpha, L)$ in a $C^*$-algebra $B$, then there is a Toeplitz representation $(\psi_V,\pi)$ of $M_L$ in $B$ such that $\psi_V(q(a))=\pi(a)V$.
\end{lemma}

\begin{proof}
We define ${\theta}:A_L\to B$ by ${\theta}(a)=\pi(a)V$. Then $\theta$ is linear, and for $a\in A$ we have
\[
{\|{\theta}(a)\|}^2 = \|{(\pi(a)V)}^*i(a)V\|= \|\pi(L(a^*a))\|\le \|L(a^*a)\|= \|{\langle a,a\rangle}_L\|,
\]
so $\theta$ is bounded for the semi-norm on $A_L$ and extends to a bounded map ${\psi}_V:M_L\to B$. To see that $({\psi}_V,\pi)$ is a Toeplitz representation of $M_L$, we let $a,b,c\in A$ and compute:
\begin{gather*}
{\psi}_V(q(b)\cdot a)={\psi}_V(q(b\alpha(a)))=\pi(b\alpha(a))V=\pi(b)V\pi(a)={\psi}_V(q(b))\pi(a),\\
{{\psi}_V(q(b))}^*{\psi}_V(q(c))={(\pi(b)V)}^*\pi(c)V=V^*\pi(b^*c)V=\pi(L(b^*c))=\pi({\langle q(b),q(c)\rangle}_L),\ \mbox{ and }\\
{\psi}_V(a\cdot q(b))={\psi}_V(q(ab))=\pi(ab)V=V\pi(a)\pi(b)V=\pi(a){\psi}_V(q(b)).\qedhere
\end{gather*}
\end{proof}

\begin{lemma}\label{essreducing}
Suppose that $(\psi,\pi)$ is a Toeplitz representation of $M_L$ on a Hilbert space $\HH$. Then the essential subspace $\KK:=\overline{\newspan}\{\pi(a)h:a\in A,h\in\HH\}$ is reducing for $(\psi,\pi)$, and we have  $\pi{|}_{{\KK}^{\perp}}=0$ and $\psi{|}_{{\KK}^{\perp}}=0$.
\end{lemma}

\begin{proof}
It is standard that $\KK$ is reducing for $\pi$ and $\pi|_{\KK}=0$, so we need to show that
$\KK$ and ${\KK}^{\perp}$ are invariant under $\psi$. 
Let $m\in M_L$ and $k\in\KK$. Since $M_L$ is essential, the Cohen factorisation theorem (as in \cite[Proposition~2.33]{rw}, for example) allows us to factor $m=a\cdot m'$. Then $\psi(m)k=\psi(a\cdot m')k=\pi(a)\psi(m')k$ belongs to $\KK$, so $\KK$ is invariant under $\psi$. Next, for $m\in M_L$ and $h\in{\KK}^{\perp}$, we have
\[
{\|\psi(m)h\|}^2 = (\psi(m)h\,|\,\psi(m)h)
= ({\psi(m)}^*\psi(m)h\,|\,h)
= (\pi({\langle m,m\rangle}_L)h\,|\,h)
= 0,
\]
because $\pi({\langle m,m\rangle}_L)k'\in\KK$. Hence $\psi(m)k'=0$ for all $k'\in\KK^{\perp}$, which implies that $\KK^{\perp}$ is invariant under $\psi$ and that $\psi(m){|}_{{\KK}^{\perp}}=0$.
\end{proof}

\begin{proof}[Proof of Proposition~\ref{the Toeplitz algebra}]
By Lemma~\ref{defpsiV}, there is a Toeplitz representation $(\psi_S,i)$ of $M_L$ in $\TT(A,\alpha,L)$. We will use \cite[Proposition~1.3]{fr} to prove that $(\TT(A,\alpha,L),{\psi}_S,i)$ has the universal property which characterises $(\TT(M_L),i_{M_L},i_A)$. Since $\TT(A,\alpha,L)$ is generated by $i(A)\cup i(A)S$, it is generated by $i(A)\cup \psi_S(M_L)$. Next, let $(\psi,\pi)$ be a Toeplitz representation of $M_L$ in $B$, and aim to prove that there is a representation $\psi\times\pi$ of $\TT(A,\alpha,L)$ such that $(\psi\times\pi)\circ  i=\pi$ and $(\psi\times\pi)\circ  \psi_S=\psi$.

We choose a a faithful nondegenerate representation $\rho:B\rightarrow B(\HH)$, and consider the Toeplitz representation $({\psi}_0,{\pi}_0):=(\rho\circ\psi,\rho\circ\pi)$. Lemma~\ref{essreducing} implies that the restriction $({\psi}_0|_{\KK},{\pi}_0|_{\KK})$ to the essential subspace $\KK$ of $\pi_0$ is a Toeplitz representation of $M_L$ on $\KK$ with ${\pi}_0|_{\KK}$ nondegenerate, so Lemma~\ref{corres. between T-c repns and T repns p1} gives a Toeplitz-covariant representation $({\pi}_0|_{\KK},V)$ of $(A,\alpha,L)$ on $\KK$, and the universal property of $\TT(A,\alpha,L)$ gives a nondegenerate representation ${\pi}_0|_{\KK}\times V:\TT(A,\alpha,L)\rightarrow B(\KK)$ satisfying $({\pi}_0|_{\KK}\times V)\circ i={\pi}_0|_{\KK}$ and $\overline{({\pi}_0|_{\KK}\times V)}(S)=V$. The representation
\[
\mu:=({\pi}_0|_{\KK}\times V)\oplus 0:\TT(A,\alpha,L)\rightarrow B(\HH)
\]
then satisfies $\mu\circ i={\pi}_0=\rho\circ\pi$, and $\mu\circ{\psi}_S={\psi}_0=\rho\circ\psi$. Since $\TT(A,\alpha,L)$ is generated by $i(A)\cup{\psi}_S(M_L)$, and the range of $\rho$ is closed, we have $\range\mu\subset\range\rho$, and the homomorphism $\psi\times\pi:={\rho}^{-1}\circ\mu$ has the required properties. The result now follows from \cite[Proposition~1.3]{fr}.
\end{proof}

\begin{cor}\label{i is injective}
The map $i:A\rightarrow \TT(A,\alpha,L)$ is injective. The map $i_A:A\to \TT(M_L)$ is nondegenerate, and the canonical Toeplitz representation  $(i_{M_L},i_A)$ is universal for Toeplitz representations $(\psi,\pi)$ in which $\pi$ is nondegenerate.
\end{cor}

\begin{proof}
Proposition~1.3 of \cite{fr} implies that $i_A:A\rightarrow\TT(M_L)$ is injective, and so therefore is $i=(\psi_S\times i)^{-1}\circ i_A$.  On the other hand, $i$ is nondegenerate, and so is $i_A=(\psi_S\times i)\circ i$. The last statement now follows from the universal property of $(\TT(M_L),i_{M_L},i_A)$.
\end{proof}

\section{The crossed product}\label{The crossed product}

Suppose that $(A,\alpha,L)$ is an Exel system, $(i,S)$ is the canonical Toeplitz representation of $(A,\alpha,L)$ in $\TT(A,\alpha,L)$, and $(\psi_S,i)$ is the Toeplitz-covariant representation of Proposition~\ref{the Toeplitz algebra}. Following \cite{e1}, we say that a pair $(i(a),k)$ in $\TT(A,\alpha,L)$ is a \emph{redundancy} if $k\in \overline{i(A)SS^*i(A)}$ and
$i(a)i(b)S=ki(b)S$ for all $b\in A$. As in \cite[Lemma~3.5]{br}, $(i(a),k)$ is a redundancy if and only if $a\in \phi^{-1}(\KK(M_L))$ and $k=(\psi_S,i)^{(1)}(\phi(a))$. 

Following \cite{e1}, we define the crossed product $A\rtimes_{\alpha,L}\N$ to be the quotient of $\TT(A,\alpha,L)$ by the ideal $I(A,\alpha,L)$ generated by the elements
$i(a)-k$ such that $(i(a),k)$ is a redundancy and $a\in\overline{A\alpha(A)A}$. As in \cite[Corollary~3.6]{br}, we write $K_\alpha:=\overline{A\alpha(A)A}\cap\phi^{-1}(\KK(M_L))$, and then $I(A,\alpha,L)$ is the ideal generated by the elements $i(a)-(\psi_S,i)^{(1)}(\phi(a))$ for $a\in K_\alpha$. We write $Q$ for the quotient map of $\TT(A,\alpha,L)$ onto $A\rtimes_{\alpha,L}\N$.

As in \cite[Proposition~3.6]{br}, the crossed product $(A\rtimes_{\alpha, L}\N,Q\circ i,\overline{Q}(S))$ is universal for Toeplitz representations $(\pi,V)$ of $(A,\alpha,L)$ which are \emph{covariant} in the sense that
\[
\pi(a)={({\psi}_V,\pi)}^{(1)}(\phi(a))\ \text{  for all $a\in K_{\alpha}$.}
\]
Then, extending \cite[Proposition~3.10]{br}, we have:
 
\begin{thm}\label{our redundancies}
For every Exel system $(A,\alpha,L)$, there is an isomorphism $\theta$ of the relative Cuntz-Pimsner algebra $\OO(K_\alpha,M_L)$ onto $A\rtimes_{\alpha,L}\N$ such that $\theta\circ k_A=Q\circ i$ and $\theta\circ k_{M_L}=Q\circ \psi_S=\psi_{\overline{Q}(S)}$.
\end{thm}

\begin{proof}
We begin by observing that $Q(\psi_S(q(a)))=Q(\pi(a))\overline{Q}(S)$, so $Q\circ\psi_S$ coincides with the representation $\psi_{\overline{Q}(S)}$ asssociated to $(Q\circ\pi, \overline{Q}(S))$. 
We prove that $(A{\rtimes}_{\alpha,L}\N,{\psi}_{\overline{Q}(S)},Q\circ i)$ has the universal property which characterises $(\OO(K_\alpha,M_L),k_{M_L},k_A)$.  Since $\pi(A)\cup\psi_S(M_L)$ generates the Toeplitz algebra, 
$Q\circ\pi(A)\cup{\psi}_{\overline{Q}(S)}(M_L)$ generates $A{\rtimes}_{\alpha,L}\N$.

Suppose that $(\psi,\pi)$ is a Toeplitz representation of $M_L$ in a $C^*$-algebra $B$ which is coisometric on $K_{\alpha}$. As in the proof of Proposition~\ref{the Toeplitz algebra}, we choose a faithful nondegenerate representation $\rho:B\rightarrow B(\HH)$, and consider the Toeplitz representation $({\psi}_0,{\pi}_0):=(\rho\circ\psi,\rho\circ\pi)$ of $M_L$ on $\HH$. The identity ${(\rho\circ\psi,\rho\circ\pi)}^{(1)}=\rho\circ{(\psi,\pi)}^{(1)}$ (see \cite[Section 1]{fmr}) implies that $({\psi}_0,{\pi}_0)$ is coisometric on $K_{\alpha}$. Now we restrict $(\psi_0,\pi_0)$ to the essential subspace $\KK$ for $\pi_0$, and, as in the proof of Proposition~\ref{the Toeplitz algebra}, we get a Toeplitz-covariant representation $(\pi_0|_{\KK},V)$. A straightforward calculation shows that $(\psi_0|_{\KK},\pi_0|_{\KK})^{(1)}(T)=(\psi_0,\pi_0)^{(1)}(T)|_{\KK}$ for $T=\Theta_{m,n}$, and this extends by linearity and continuity to $T\in \KK(M_L)$. Thus
\begin{align*}
{({\psi}_V,{\pi}_0|_{\KK})}^{(1)}(\phi(a))(k) &= {({\psi}_0|_{\KK},{\pi}_0|_{\KK})}^{(1)}(\phi(a))(k)
= {({\psi}_0,{\pi}_0)}^{(1)}|_{\KK}(\phi(a))(k)\\
&= {({\psi}_0,{\pi}_0)}^{(1)}(\phi(a))(k)
= {\pi}_0(a)(k)= {\pi}_0|_{\KK}(a)(k),
\end{align*}
so $({\pi}_0|_{\KK},V)$ is a covariant representation of $(A,\alpha,L)$, and gives a representation ${\pi}_0|_{\KK}\times V$ of $A\rtimes_{\alpha,L}\N$. Then $\nu:=\rho^{-1}\circ\big(({\pi}_0|_{\KK}\times V)\oplus 0\big)$ satisfies $\nu\circ(Q\circ i)=\pi$ and $\nu\circ\psi_{\overline{Q}(S)}=\psi$. 

The result now follows from \cite[Proposition~1.3]{fmr}.
\end{proof}

From now on, we use the isomorphism of Theorem~\ref{our redundancies} to identify $A\rtimes_{\alpha,L}\N$ with $\OO(K_\alpha,M_L)$, and we write $(k_{M_L},k_A)$ for the canonical Toeplitz representation of $M_L$ in $A\rtimes_{\alpha,L}\N=\OO(K_\alpha,M_L)$.

For systems $(C_0(T),\alpha,L)$ arising from classical systems $(T,\tau)$, $\phi:C_0(T)\to \LL(M_L)$ has range in $\KK(M_L)$. To see this, it suffices to prove that $\phi(f)\in \KK(M_L)$ for every $f\in C_c(T)$. Choose a finite cover $\{U_i\}$ of $\supp f$ by relatively compact open sets such that $\tau|{U_i}$ is one-to-one, and let $\{p_i\}$ be a partition of unity subordinate to $\{U_i\}$. Define $g_i=(|\tau^{-1}(\tau(t))|\rho_i(t))^{1/2}$. Then for $h\in C_0(T)$ we have
\[
(\Theta_{fg_i,g_i}h)(t)=f(t)g_i(t)\frac{1}{|\tau^{-1}(\tau(t))|}\sum_{\tau(s)=\tau(t)}g_i(s)h(s)=f(t)\rho_i(t)h(t),
\]
so 
\begin{equation}\label{LHScpct}
\phi(f)=\sum_{i}\Theta_{fg_i,g_i}
\end{equation}
belongs to $\KK(M_L)$. Since $\alpha$ is nondegenerate, $\overline{A\alpha(A)A}=A$, and $K_\alpha=A$. Thus:

\begin{cor}\label{cpclassical}
Suppose that $(C_0(T),\alpha,L)$ arises from a classical system $(T,\tau)$, as in \textnormal{\S\ref{subsectcomm}}. Then $(C_0(T)\rtimes_{\alpha,L}\N,k_{M_L},k_A))$ is the Cuntz-Pimsner algebra $(\OO(M_L),k_{M_L},k_A)$.
\end{cor}

Next, we recall from \cite{br} that if $I$ is an ideal in $A$, the transfer operator $L$ is \emph{faithful on $I$} of $A$ if $a\in I\mbox{ and }L(a^*a)=0\Longrightarrow a=0$, and \emph{almost faithful on $I$} if
\[
a\in I\mbox{ and }L({(ab)}^*ab)=0 \text{ for all } b\in A\Longrightarrow a=0.
\]
The arguments of Theorem~4.2 and Corollary~4.3 of \cite{br} give the following results on the injectivity of $k_A:A\to A\rtimes_{\alpha,L}\N$. The examples in \cite[\S4]{br} show that they are sharp.

\begin{thm}\label{injectivity of j_A for non-comm}
Suppose $(A,\alpha,L)$ is an Exel system. Then $Q\circ i:A\to A{\rtimes}_{\alpha,L}\N$ is injective if and only if $L$ is almost faithful on $K_{\alpha}:=\overline{A\alpha(A)A}\cap J(M_L)$. 
\end{thm}

\begin{cor}\label{injectivity of j_A}
Suppose $(A,\alpha,L)$ is an Exel system with $A$ commutative. Then $Q\circ i:A\to A{\rtimes}_{\alpha,L}\N$ is injective if and only if $L$ is faithful on $K_{\alpha}$.
\end{cor}

\begin{cor}\label{ansexelQ}
Suppose that $(C_0(T),\alpha,L)$ arises from a classical system $(T,\tau)$, as in \textnormal{\S\ref{subsectcomm}}. Then the canonical map $k_{A}$ of $C_0(T)$ into $C_0(T)\rtimes_{\alpha,L}\N=\OO(M_L)$ is injective.
\end{cor}

\begin{proof}
We just need to observe that
\[
L(f^*f)=0\Longrightarrow \sum_{\tau(s)=t}|f(s)|^2=0\text{ for all $t$}\Longrightarrow |f(s)|^2=0\text{ for all $s$}\Longrightarrow f=0.\qedhere
\]
\end{proof}

\section{Graph algebras as Exel crossed products}\label{realisegralg}

Our next theorem says that many graph algebras can be viewed as Exel crossed products associated to the classical system $(E^\infty,\sigma)$. Recall that in this case $M_L$ is the completion of a copy $\{q(f):f\in C_c(E^\infty)\}$ of $C_c(E^\infty)$.

\begin{thm}\label{T:GAasECP}
Let $E$ be a locally finite directed graph with no sources, and define $c:E^0\to [0,\infty)$ by $c(v)=|s^{-1}(v)|$. Then the elements
\begin{equation}\label{forms for ck family}
S_{e} := \sqrt{c(s(e))}k_{M_L}(q(\chi_{Z ( e )}) )\ \text{ and }\ P_{v} := k_{A}( \chi_{Z ( v )} )
\end{equation}
form a Cuntz-Krieger $E$-family, and the homomorphism $\pi_{S,P}:C^*(E)\to C_0(E^\infty)\rtimes_{\alpha,L}\N$ is an isomorphism. For $\mu\in E^n$, we have
\begin{equation}\label{relateSPtopipsi}
k_{A}(\chi_{Z(\mu)})=S_\mu S_\mu^*\ \text{ and }\ k_{M_L}(q(\chi_{Z(\mu)}))=c(s(\mu_1))^{-1/2}S_{\mu}S_{\mu_2\cdots\mu_n}^*.
\end{equation}
\end{thm}

To make our calculations more legible we are going to drop the map $q:C_c(E^{\infty})\to M_L$ from our notation. We will use the next lemma several times.

\begin{lemma}\label{L:leftaction}
For $\mu\in E^*$ with $|\mu|\geq 1$ we have  
\[
\phi( \chi_{Z ( \mu )} ) = c(s(\mu_1))\Theta_{ \chi_{Z ( \mu )}, \chi_{Z ( \mu_{1} )} }=c(s(\mu_1))\Theta_{ \chi_{Z ( \mu_1)}, \chi_{Z ( \mu )} }.\]
\end{lemma}

\begin{proof}
We let $f \in C_c(E^\infty)$ and $\xi \in E^\infty$, and compute:
\begin{align}
c(s(\mu_1))(&\Theta_{ \chi_{Z ( \mu )},\chi_{Z ( \mu_{1} )} } ( f ) ) ( \xi )
= c(s(\mu_1))( \chi_{Z ( \mu )} \cdot \langle \chi_{Z ( \mu_{1} )}, f \rangle_{L} ) ( \xi )\notag\\
&=c(s(\mu_1)) \chi_{Z ( \mu )} ( \xi ) \langle \chi_{Z ( \mu_{1} )}, f \rangle_{L} ( \sigma ( \xi ) )\notag\\ 
&= c(s(\mu_1))\chi_{Z ( \mu )} ( \xi ) c(r(\sigma(\xi)))^{-1}\textstyle{\sum_{s(e)=r(\sigma(\xi))}} \chi_{Z ( \mu_{1} )} ( e\sigma(\xi) ) f (e\sigma(\xi)).\label{formforrank1}
\end{align}
This vanishes unless $\xi=\mu\xi'$,  and  then $e=\mu_1=\xi_1$ is the only edge which gives a non-zero summand: then $e\sigma(\xi)=\xi$, $r(\sigma(\xi))=s(\xi_1)=s(\mu_1)$ and \eqref{formforrank1} is $(\chi_{Z(\mu)}f)(\xi)=(\phi(\chi_{Z(\mu)})(f))(\xi)$. The second formula follows from a similar calculation.
\end{proof}

\begin{proof}[Proof of Theorem~\ref{T:GAasECP}]
The projections $\{P_v\}$ are mutually orthogonal because the $\chi_{Z(v)}$ are. Next, observe that $\langle \chi_{Z(e)},\chi_{Z(e)}\rangle_L=L(\chi_{Z(e)})=c(s(e))^{-1}\chi_{Z(s(e))}$, so
\[
S_e^*S_e=c(s(e))k_{A}\big(\langle \chi_{Z ( e )},\chi_{Z ( e )}\rangle_L\big)=k_{A}(\chi_{Z(s(e))})=P_{s(e)}.
\]
To verify the Cuntz-Krieger relation at a vertex $v$, we compute using covariance and Lemma~\ref{L:leftaction} for $\mu=e$:
\begin{align}\label{calcCK}
\textstyle{\sum_{r ( e ) = v}} S_{e} S_{e}^{*}
&= \textstyle{\sum_{r ( e ) = v}} c(s(e))k_{M_L}( \chi_{Z ( e )} ) k_{M_L} ( \chi_{Z ( e )} )^{*}
\\&= \textstyle{\sum_{r( e ) = v}} (k_{M_L},k_{A} )^{( 1 )} ( c(s(e))\Theta_{ \chi_{Z ( e )}, \chi_{Z ( e )} } )\notag
\\&= \textstyle{\sum_{r ( e ) = v}} (k_{M_L},k_{A})^{( 1 )} ( \phi( \chi_{Z ( e )} ) )\notag
\\&= k_{A}( \textstyle{\sum_{r ( e ) = v}} \chi_{Z ( e )} )\notag
\\&= k_{A}( \chi_{Z ( v )} )
= P_{v}.\notag
\end{align}
So $\{S_e,P_v\}$ is a Cuntz-Krieger $E$-family, and gives a homomorphism $\pi_{S,P}:C^*(E)\to \OO(M_L)$. Since $k_{A}$ is faithful (Corollary~\ref{ansexelQ}), the projections $p_v$ are all non-zero, and the gauge-invariant uniqueness theorem for graph algebras implies that $\pi_{S,P}$ is faithful. 

To see that $\pi_{S,P}$ is surjective, it suffices to show that every $k_{M_L}(\chi_{Z(\mu)})$ and every $k_{A}(\chi_{Z(\mu)})$ belongs to $\range\pi_{S,P}$. We prove by induction that $k_{M_L}(\chi_{Z(\mu)})\in\range \pi_{S,P}$ for every $\mu\in E^{n+1}$ and $k_{A}(\chi_{Z(\nu)})\in\range \pi_{S,P}$ for every $\nu\in E^{n}$. This is true for $n=0$ by definition of $S_e$ and $P_v$. Suppose it is true for $n=k$, and let $\nu\in E^{k+1}$ and $\mu\in E^{k+2}$. Using Lemma~\ref{L:leftaction}, we have
\begin{align}\label{indhyppi}
k_{A}( \chi_{Z ( \nu )} )
&= (k_{M_L},k_{A})^{( 1 )} ( \phi( \chi_{Z ( \nu )} ) )
= ( k_{M_L},k_{A})^{( 1 )} ( \Theta_{ \chi_{Z ( \nu )}, \chi_{Z ( \nu_{1} )} } )\\
&=k_{M_L}( \chi_{Z ( \nu )} ) k_{M_L}( \chi_{Z ( \nu_{1} )} )^{*},\notag
\end{align}
which belongs to $\range\pi_{S,P}$ by the inductive hypothesis. Next, we use the inductive hypothesis on $k_{M_L}$ and \eqref{indhyppi} (for $\nu=\mu_{2} \cdots \mu_{n + 2}$) to see that
\begin{align}\label{psiin}
k_{M_L}( \chi_{Z ( \mu )} )&= k_{M_L}( \chi_{Z ( \mu_{1} )} \alpha(\chi_{Z ( \mu_{2} \cdots \mu_{n + 2} )}) )\\
&= k_{M_L}( \chi_{Z ( \mu_{1} )} \cdot \chi_{Z ( \mu_{2} \cdots \mu_{n + 2} )}  )\notag\\
&= k_{M_L}( \chi_{Z ( \mu_{1} )} ) k_{A} ( \chi_{Z ( \mu_{2} \cdots \mu_{n + 2} )} )\notag
\end{align}
belongs to $\range\pi_{S,P}$. Thus  $\pi_{S,P}$ is surjective.

The second formula in \eqref{relateSPtopipsi} follows from a calculation like that in \eqref{psiin}. We prove the first formula by induction on $n$. It is trivially true for $n=0$. So suppose it is true for $n=k$. Now we let $\mu\in E^{k+1}$ and calculate, using Lemma~\ref{L:leftaction} again:
\begin{align*}
k_{A}(\chi_{Z(\mu)})=k_{A}(\chi_{Z(\mu)})^2&=(k_{M_L},k_{A})^{(1)}\big(c(s(\mu_1))^2\Theta_{ \chi_{Z ( \mu_1)}, \chi_{Z ( \mu )} }\Theta_{ \chi_{Z ( \mu)}, \chi_{Z ( \mu_1)} }\big)\\
&=c(s(\mu_1))^2k_{M_L}(\chi_{Z ( \mu_1)})k_{M_L}(\chi_{Z ( \mu )})^*k_{M_L}(\chi_{Z ( \mu)})k_{M_L}(\chi_{Z ( \mu_1)})^*\\
&=c(s(\mu_1))^2k_{M_L}(\chi_{Z ( \mu_1)})k_{A}\big(\langle\chi_{Z ( \mu )},\chi_{Z ( \mu)}\rangle_L\big)k_{M_L}(\chi_{Z ( \mu_1)})^*\\
&=c(s(\mu_1))S_{\mu_1}k_{A}(L(\chi_{Z(\mu)}))S_{\mu_1}^*.
\end{align*}
A quick calculation on the side shows that $L(\chi_{Z(\mu)})=c(s(\mu_1))^{-1}\chi_{Z(\mu_2\cdots\mu_{k+1})}$, so the inductive hypothesis implies that
\[
k_{A}(\chi_{Z(\mu)})=S_{\mu_1}(S_{\mu_2\cdots\mu_{k+1}}S_{\mu_2\cdots\mu_{k+1}}^*)S_{\mu_1}^*=S_\mu S_\mu^*.\qedhere
\]
\end{proof}

\section{Simplicity for classical systems}\label{The system (C_0(X),alpha,L)}

To find criteria for the simplicity of crossed products $C_0(T)\rtimes_{\alpha,L}\N=\OO(M_L)$, we want to use Katsura's general theory of topological graphs \cite{k,k2} (as in \cite{er}): to study the classical system $(T,\tau)$, we use the topological graph $E=(T,T,\tau,\id)$. The bimodule $M_L$ is not quite the same as the bimodule $C_\tau(E)$ appearing in \cite{k}, but it is isomorphic to it (this too has been noticed elsewhere, including \cite{im}).  Indeed, both bimodules can be viewed as completions of $C_c(T)$, the only difference being that the inner product $\langle \cdot,\cdot\rangle_E$ in $C_\tau(E)$ satisfies
\[
\langle f,g\rangle_E=\sum_{\tau(s)=t}\overline{f(s)}g(t) =|\tau^{-1}(t)|\langle f,g\rangle_L(t).
\]
The formula $U(f)(t)=\sqrt{|\tau^{-1}(\tau(t))|}f(t)$ defines a $C_0(T)$--$C_0(T)$ bimodule homomorphism $U$ from $C_c(T)\subset C_\tau(E)$ to $C_c(T)\subset M_L$ such that $\langle Uf,Ug\rangle_L=\langle f,g\rangle_E$. Thus $U$ extends to an isomorphism of Hilbert bimodules, and the Cuntz-Pimsner algebras $\OO(E):=\OO(C_\tau(E))$ and $C_0(T)\rtimes_{\alpha,L}\N=\OO(M_L)$ are isomorphic. Thus we can use Katsura's results to study $C_0(T)\rtimes_{\alpha,L}\N$. 

We next describe the faithful representations of $C_0(T)\rtimes_{\alpha,L}\N$. Following Exel-Vershik \cite{ev}, we say that $( T, \tau)$ is \emph{topologically free} if the sets $H_{m,n}:=\{t \in T \mid \tau^{m} ( t ) = \tau^{n} ( t )\}$ have empty interior for every $m\not=n\in\N$. The next result extends Theorem~10.3 of \cite{ev}.

\begin{thm}\label{CKunique}
Suppose that $\tau:T\to T$ is a proper local homeomorphism such that $( T, \tau )$ is topologically free, and $(\psi,\pi)$ is a covariant representation of $M_L$ such that $\pi$ is faithful. Then $\psi\times\pi$ is faithful on $\OO(M_L)=C_{0} ( T ) \rtimes_{\alpha, L} \N$.
\end{thm}

We need to relate the Exel-Vershik notion of topological freeness which we are using to the one used in \cite{k}, and then Theorem~\ref{CKunique} follows immediately from \cite[Theorem~5.10]{k}.

\begin{lemma}\label{L:topfree}
The system $( T, \tau )$ is topologically free if and only if the topological graph $E=(T,T,\tau,\id)$ is topologically free.
\end{lemma}

\begin{proof}
Suppose that $(T,\tau)$ is topologically free. We need to show that the set of base points of loops without entries has empty interior. The loops in $E$ are the paths $t\tau(t)\cdots\tau^n(t)$ with $t=\tau^{n}(t)$; an entry would be an element $s\in E^1=T$ which has the same range as some $\tau^i(t)$ but is not itself $\tau^i(t)$, and since the range map in $E$ is the identity, there is no such $s$. So the set of base points of loops without entries is $\bigcup_{n=1}^\infty\{t:t=\tau^n(t)\}=\bigcup_{n=1}^\infty H_{0,n}$. Since each $H_{0,n}$ has no interior, the Baire category theorem for locally compact spaces \cite[Theorem~48.2]{munk} implies that $\bigcup_{n=1}^\infty H_{0,n}$ has empty interior.

Now suppose that $( T, \tau )$ is not topologically free. Then there exists $( m, n ) \in \N^{2}$ with $m > n$ such that $H_{m,n}$ contains an open set $V$. Since local homeomorphisms are open mappings, $\tau^{n} ( V )$ is open, and since $\tau^{n} (V) \subset H_{0,m-n}$, the set $\bigcup_{k=1}^\infty H_{0,k}$ has interior.
\end{proof}

\begin{example}\label{topfree=L}
Suppose that $E$ is a locally finite graph with no sources. We claim that the system $(E^\infty,\sigma)$ is topologically free if and only if every cycle in $E$ has an entry. 

First suppose that $(E^\infty,\sigma)$ is topologically free, and $\mu\in E^n$ is a cycle. Then $\mu\mu\mu\cdots$ belongs to $H_{0,n}$. Since $H_{0,n}$ has empty interior, the set $Z(r(\mu))$ cannot be contained in $H_{0,n}$, and there exists $\xi$ with $r(\xi)=r(\mu)$ but $\xi\not=\sigma^n(\xi)$. Then $\xi\not=\mu\mu\mu\cdots$, and the first $\xi_k$ which is not equal to $(\mu\mu\mu\cdots)_k$ is an entry to $\mu$.

Conversely, suppose every cycle in $E$ has an entry. We fix $m<n$, and aim to show that $H_{m,n}$ has empty interior. If $H_{m,n}$ is empty, this is trivially true, so suppose there exists $\xi\in H_{m,n}$. Then $\mu:=\xi_{m+1}\cdots\xi_{n}$ has $r(\mu)=s(\mu)$, hence contains a cycle, hence has an entry, say $e$ with $r(e)=r(\mu_j)$ but $e\not=\mu_j$. Choose $\eta\in E^\infty$ with $r(\eta)=s(e)$. Because $\xi$ is in $H_{m,n}$, $\xi_{m+k(n-m)+j}=\xi_m+j=\mu_j$ for every $k\in\N$, and then $\zeta^{(k)}:=\xi_1\cdots\xi_{m+k(n-m)+j-1}e\eta$ is a sequence in $E^\infty\setminus H_{m,n}$ which converges to $\xi$. So no point of $H_{m,n}$ is an interior point, and the claim is proved.  

The first formula in \eqref{relateSPtopipsi} shows that if $\{T,Q\}$ is a Cuntz-Krieger family on Hilbert space, then the corresponding covariant representation $(\theta,\rho)$ of $M_L$ satisfies $\rho(\chi_{Z(\mu)})=T_\mu T_\mu^*$. Theorem~\ref{CKunique} says that $\theta\times \rho$ is faithful if and only if $\rho$ is faithful on $C(E^\infty)$. On the face of it, this is weaker than the Cuntz-Krieger uniqueness theorem, which says that $\pi_{T,Q}$ is faithful if and only if $Q_v\not=0$ for every $v\in E^0$, and implies that $\theta\times \rho$ is faithful if and only if every $Q_v\not=0$. However, $C(E^\infty)$ is the direct limit of the subalgebras $D_n=\clsp\{\chi_{Z(\mu)}:|\mu|=n\}$. If every $Q_v$ is non-zero, then every $S_\mu S_\mu^*=Q_{s(\mu)}\not=0$, the projections $\{S_\mu S_\mu^*:|\mu|=n\}$ are mutually orthogonal and non-zero, $\rho$ is faithful on each $D_n$, and hence also on the direct limit $C(E^\infty)$ by \cite[Proposition~A.8 ]{r}. So Theorem~\ref{CKunique}, as it applies to $(E^\infty,\sigma)$, is equivalent to the Cuntz-Krieger theorem for $E$. 
\end{example}

Next we characterise the systems $(T,\tau)$ for which $C_{0} ( T ) \rtimes_{\alpha, L} \N$ is simple. Again following \cite{ev}, we say that a subset $Y$ of $T$ is \emph{invariant}\footnote{In \cite{ev}, they define $Y$ to be invariant if $t\in Y\text{ and }\tau^m(s)=\tau^n(t)\Longrightarrow s\in Y$, and claim that this is equivalent to $\tau^{-1}(Y)\subset Y$. We think they inadvertently omitted the extra condition $\tau(Y)\subset Y$, since it has to be there: for example, with $\tau:\T\to \T$ given by $\tau(z)=z^2$, the set $Y=\{\exp(2\pi i k2^{-n}):k\in \N,n\geq 2\}$ satisfies $\tau^{-1}(Y)\subset Y$ but is not invariant.} if we have $\tau(Y)\subset Y$ and $\tau^{-1}(Y)\subset Y$, and that $(T,\tau)$ is \emph{irreducible} if the only closed invariant subsets are $\emptyset$ and $T$. Our version of \cite[Theorem 11.2]{ev} differs from that theorem in that we need to assume topological freeness as well as irreducibility. When $\tau$ is a covering map on an infinite compact space, irreducibility implies topological freeness \cite[Proposition~11.1]{ev}, but this is not true for locally compact $T$, as our later examples show.

\begin{thm}\label{T:simple}
Suppose that $\tau:T\to T$ is a proper local homeomorphism. Then $C_0(T)\rtimes_{\alpha,L}\N$ is simple if and only if $(T,\tau)$ is topologically free and $\tau$ is irreducible.
\end{thm}

\begin{proof}
Lemma~\ref{L:topfree} says that $(T,\tau)$ is topologically free if and only if $E=(T,T,\tau,\id)$ is, and it is easy to see that $E$ is minimal in the sense of \cite{k2} if and only if $(T,\tau)$ is irreducible, so the result follows immediately from  Theorem~8.12 of \cite{k2}.
\end{proof}

\begin{example}
Suppose that $E$ is a locally finite graph with no sources. We claim that  $(E^\infty,\sigma)$ is irreducible if and only if $E$ is cofinal. This claim and the one in Example~\ref{topfree=L} say that the criteria in Theorem~\ref{T:simple} to $(E^\infty,\sigma)$ reduce to the known criteria for simplicity of $C^*(E)=C_0(E^\infty)\rtimes_{\alpha,L}\N$, as in \cite[Proposition~5.1]{bprs} or \cite[Theorem~4.14]{r}.

Suppose $E$ is cofinal, and $Y$ is a nonempty open invariant subset of $E^\infty$. Let $\xi\in E^\infty$. Since $Y$ is open, it contains a cylinder set $Z(\mu)$, and cofinality implies that there exists $\nu\in E^*$ with $r(\nu)=s(\mu)$ and $s(\nu)=\xi_k$ for some $k$. Then $\eta:=\mu\nu\xi_k\xi_{k+1}\cdots$ is in $Z(\mu)\subset Y$, and since $\sigma^k(\xi)=\sigma^{|\mu|+|\nu|}(\eta)$, invariance of $Y$ implies that $\xi\in Y$. Thus $Y=E^\infty$, as required. Conversely, suppose that $(E^\infty,\sigma)$ is irreducible. Then for $v\in E^\infty$, 
\[
Y_{v} := \{\xi\in E^{\infty}:\text{there exists $\mu\in E^*$ with $r(\mu)=v$ and $s(\mu)=r(\xi_k)$ for some $k$}\}
\]
is a non-empty open invariant subset of $E^{\infty}$, hence all of $E^\infty$. But this says precisely that $v$ can be reached from every infinite path in $E$, and hence that $E$ is cofinal.
\end{example}

\section{Gauge-invariant ideals in crossed products for classical systems}\label{secgauge}

We study the gauge-invariant ideals of crossed products  associated to classical systems $(T,\tau)$ using the general theory of \cite{k2}.

\begin{lemma}\label{L:YIinv}
For every ideal $I$ of $C_{0} ( T ) \rtimes_{\alpha, L} \N$, 
\begin{equation}\label{form for I Y}
Y_{I}:= \{ t \in T : \mbox{$f ( t ) = 0$ for all $f \in C_0(T)$ such that $k_{A} ( f ) \in I$} \},
\end{equation}
is a closed invariant subset of $T$ in $( T, \tau )$.
\end{lemma}

\begin{proof}
The set $Y_I$ is the kernel of the ideal $k_A^{-1}(I)$, so it is closed. Propositions 2.5 and 2.7 of \cite{k2} say that $Y_{I}$ is an invariant subset of the topological graph $E=(T,T,\tau,\id)$, which is the same thing as invariance in $(T,\tau)$. However, it is easy to give a short direct proof. First, we suppose that $\tau(t)\in Y_I$ and aim to prove that $t\in Y_I$. Let $f\in C_0(T\setminus Y_I)$; we need to prove that $f(t)=0$. Choose $g$ such that $g(t)=1$ and $g$ has support in a neighbourhood of $t$ on which $\tau$ is one-to-one. Since $k_A(\langle g, \phi(f)g\rangle_L)=k_{M_L}(g)^*k_A(f)k_{M_L}(g)$ is in $I$, we have $\langle g, \phi(f)g\rangle_L(\tau(t))=0$, and the calculation 
\[
0=\langle g, \phi(f)g\rangle_L(\tau(t))=\frac{1}{|\tau^{-1}(\tau(t))|}\sum_{\tau(s)=\tau(t)}|g(s)|^2f(s)=|\tau^{-1}(\tau(t))|^{-1}f(t)
\]
shows that $f(t)=0$, as required.

We show that $\tau(Y_I)\subset Y_I$ by proving that $k_A(C_c(T\setminus\tau^{-1}(Y_I)))\subset I$. Let $f\in C_c(T\setminus\tau^{-1}(Y_I))$, and write $\phi(f)=\sum_{i}\Theta_{fg_i,g_i}$ as in \eqref{LHScpct}. Then for $t\in Y_I$ and each $i$, we have
\[
\langle \phi(f)g_i,\phi(f)g_i\rangle_L(t)=\langle fg_i,fg_i\rangle_L(t)=L(|fg_i|^2)(t)=\frac{1}{|\tau^{-1}(t)|}\sum_{\tau(s)=t}|fg_i|^2(s)=0,
\]
because $f$ vanishes on $\tau^{-1}(Y_I)$. Thus $k_{M_L}(fg_i)^*k_{M_L}(fg_i)=k_A(\langle fg_i,fg_i\rangle_L)$ belongs to $I$, and so does $k_{M_L}(fg_i)$. Thus 
\[
k_A(f)=(k_A,k_{M_L})^{(1)}(\phi(f))=(k_A,k_{M_L})^{(1)}\big(\textstyle{\sum_i}\Theta_{fg_i,g_i}\big)=\sum_i k_{M_L}(fg_i)k_{M_L}(fg_i)^*
\]
also belongs to $I$, as required.
\end{proof}

\begin{thm}\label{T:gaugeideal}
Suppose that $(C_0(T),\alpha,L)$ arises from a classical system $(T,\tau)$. Then the map $I \mapsto Y_{I}$ is a bijection from the set of gauge-invariant ideals of $C_{0} ( T ) \rtimes_{\alpha, L} \N$ to the set of closed invariant subsets of $T$. The inverse takes a closed invariant set $Y$ to the ideal $I_Y$ generated by $\{k_A(f):f\in C_c(T\setminus Y)\}$.
\end{thm}

\begin{proof}
Since the range map $\id$ is surjective, all the vertices in $E=(T,T,\tau,\id)$ are regular; in the notation of \cite[\S2]{k2}, $E^0_{\rg}=T$ and $E^{0}_{\sg}=\emptyset$. Thus the ``admissible pairs'' in \cite[Definition 2.3]{k2} are $( Y, \emptyset )$ for $Y$ closed and invariant in $T$. So \cite[Theorem~3.19]{k2} implies that $I \mapsto Y_{I}$ is a bijection. It remains to identify the inverse. 

Suppose $Y$ is closed and invariant, giving the admissible pair $\rho=( Y, \emptyset )$. Since $\phi:C_0(T)\to \LL(M_L)$ has range in $\KK(M_L)$ by \eqref{LHScpct}, the algebra $\FF^1$ in \cite[\S3]{k2} is $(k_A,k_{C_\tau(T)})^{(1)}(\KK(C_\tau(T)))$. Thus the ideal $J_\rho$  in \cite[Definition~3.1]{k2} is 
\[
J_\rho=\{(k_A,k_{C_\tau(T)})^{(1)}(x):x\in \ker\omega_Y:\KK(C_\tau(T))\to \KK(C_{\tau|}(Y))\}.
\]
Lemma~1.14 of \cite{k} implies that $\ker\omega_Y$ is 
\[
\KK(C_{\tau|}(T\setminus Y))=\clsp\{\Theta_{f,g}:f,g\in C_c(T\setminus Y)\},
\]
and applying $(k_A,k_{C_\tau(T)})^{(1)}$ shows that (modulo the isomorphism of $C_\tau(T)$ with $M_L$ which carries $\OO(C_\tau(T))$ onto $C_0(T)\rtimes_{\alpha, L}\N$), $J_\rho=J_Y:=\clsp\{k_{M_L}(f)k_{M_L}(g)^*:f,g\in C_c(T\setminus Y)\}$. Thus the ideal $I_\rho$ in \cite[Definition~3.3]{k2} is generated by $J_Y$. 

We now claim that the ideal generated by $J_Y$ is equal to $I_Y$. Let $f\in C_c(T\setminus Y)$, and choose $h\in C_c(T\setminus Y)$ with $h|_{\supp f}=1$. Then $f=hf$, and we have $k_{M_L}(f)=k_A(h)k_{M_L}(f)$. So each $k_{M_L}(f)k_{M_L}(g)^*\in I_Y$, and $J_Y\subset I_Y$. To see that $J_Y$ generates, let  $f\in C_c(T\setminus Y)$.  Since $(k_{M_{L}},k_{A})$ is coisometric on $A$, \eqref{LHScpct} implies that 
\[
k_{A}(f) = (k_{M_{L}},k_{A})^{(1)}(\phi (f)) = (k_{M_{L}},k_{A})^{(1)} \big(\textstyle{\sum_{i}} \Theta_{f g_{i},g_{i}}\big) = \textstyle{\sum_{i}} k_{M_{L}} (f g_{i}) k_{M_{L}}(g_{i})^{*},
\]
belongs to $J_Y$, and so $I_Y$ is contained in the ideal generated by $J_Y$.
\end{proof}

\begin{example}\label{graph ex for gii} Suppose that $E$ is a locally finite graph $E$ without sources. For each closed invariant subset $Y$ of $E^\infty$, the complement $E^\infty\setminus Y$ is open and invariant, and $H_Y:=\{r(\xi):\xi\in E^\infty\setminus Y\}$ is a hereditary and saturated subset of $E^0$, as in \cite[\S4]{r}. Conversely, if $H\subset E^0$ is saturated and hereditary, then $Y_H:=\{\xi : r(\xi_i)\in E^0\setminus H\}$ is a closed invariant subset of $E^\infty$. So Theorem~\ref{T:gaugeideal} confirms that the ideals in $C^*(E)=C(E^\infty)\rtimes_{\alpha,L}\N$ are parametrised by the saturated hereditary subsets $H$ of $E^0$. 

We want to know, however, that the ideal $I_{Y_H}$ is the ideal $I_H$ generated by the projections $\{p_v:v\in H\}$ (as in \cite[\S4]{r}, for example). When we realise $C^*(E)$ as a crossed product, the projections $p_v$ are carried into the elements $k_A(\chi_{Z(v)})$ (see Theorem~\ref{T:GAasECP}). So we need to show that $I_{Y_H}$ is generated by $\{k_A(\chi_{Z(v)}):v\in H\}$. 
Certainly each $k_A(\chi_{Z(v)})$ belongs to $I_{Y_H}$. To see that they generate, we deduce from the Stone-Weierstrass theorem that $C_0(E^{\infty}\setminus Y_H)=\overline{\newspan}\{\XX_{Z(\mu)}:Z(\mu)\subset E^{\infty}\setminus Y_H\}$. We have $Z(\mu)\subset E^{\infty}\setminus Y_H\Longleftrightarrow r(\mu)\in H$. Since $\chi_{Z(\mu)}\leq \chi_{Z(r(\mu))}$ and ideals are hereditary, this implies that $k_A(\chi_{Z(\mu)})\in I_H$ belongs to the ideal generated by the $k_A(\chi_{Z(v)})$. So the $k_A(\chi_{Z(v)})$ generate.
\end{example}

Now we want to decide when every ideal is gauge-invariant, so that Theorem~\ref{T:gaugeideal} gives a description of all the ideals in $C_0(T)\rtimes_{\alpha,L}\N$. We say that $t\in T$ is \emph{periodic} if there exists $n\ge 1$ such that $\tau^n(t)=t$. The smallest such $n$ is called the {\em period}.

\begin{thm}\label{T:GIper}
Suppose that $(C_0(T),\alpha,L)$ arises from a classical system $(T,\tau)$. Then every ideal of $C_{0} ( T ) \rtimes_{\alpha, L} \N$ is gauge-invariant if and only if every periodic point $t$ is a cluster point of $\tau^{-\N}(t):=\bigcup_{n\ge 0}\tau^{-n}(t)$.
\end{thm}

\begin{proof}
Katsura proved in \cite[Theorem 7.6]{k2} that every ideal of $C_{0} ( T ) \rtimes_{\alpha, L} \N$ is gauge-invariant if and only if the topological graph $E=(T,T,\tau,\id)$ is what he calls ``free,'' so we need to reconcile this notion of freeness with our condition. 

For each $t\in T$ the set $\Orb^+(t)$ in \cite[Definition~4.1]{k2} is $\tau^{-\N}(t)$. Condition (ii) of \cite[Definition~7.1]{k2} holds trivially for $E$ because the range map $\id$ is one-to-one, so $t\in T$ is periodic and isolated in $\tau^{-\N}(t)$ if and only if $t$ is an element of the set $\Per(E)$ in \cite[Definition~7.1]{k2}. So our condition says precisely that $\Per(E)$ is empty, which is freeness. 
\end{proof}

\begin{example}
A directed graph $E$ satisfies Condition (K) if for every $v \in E^{0}$ either there is no cycle based at $v$ or there are two distinct return paths based at $v$. We claim that a locally finite graph $E$ with no sources satisfies (K) if and only if every periodic point $\xi\in E^{\infty}$ is a cluster point of $\sigma^{-\N}(\xi)$. Then Theorem~\ref{T:GIper} implies that all the ideals of $C^*(E)$ are gauge invariant if and only if $E$ satisfies (K), as in \cite[Corollary~3.8]{bhrs}. 

Suppose that $E$ satisfies Condition (K) and $\xi \in E^{\infty}$ is periodic with period $n$. We show that for each $\mu \in E^{*}$ with $\xi \in Z ( \mu )$ we have $Z ( \mu ) \cap (\sigma^{-\N}(\xi)\setminus\{\xi\})\not=\emptyset$. We know there is a cycle in $E$ based at $s ( \xi_{n} )$.  Let $1 \le k \le n$ be the largest integer such that $r ( \xi_{k} ) = s ( \xi_{n} )$. Then $\xi_{k} \cdots \xi_{n}$ is a return path in $E$ based at $s ( \xi_{n} )$ and $E$ satisfies (K), so there is a distinct return path $\eta_{1} \cdots \eta_{m}$ based at $s ( \xi_{n} )$.  Choose $j \ge 1$ such that $jn \ge | \mu |$. Then $\lambda := \xi_{1} \cdots \xi_{jn} \xi_{1} \cdots \xi_{k-1} \eta_{1} \cdots \eta_{m} \xi \in Z ( \mu ) \cap (\sigma^{-\N}(\xi)\setminus\{\xi\})$.

Conversely, suppose every periodic point $\xi\in E^{\infty}$ is a cluster point of $\sigma^{-\N}(\xi)$, and that  $\mu$ is a cycle in $E$ based at $v$. Then $\xi := \mu \mu \mu \cdots$ is a periodic point in $E^{\infty}$, and there exists $\eta \in Z ( r ( \xi ) ) \cap (\sigma^{-\N}(\xi)\setminus\{\xi\})$. Let $m \ge 1$ be the smallest integer such that $\sigma^{m} ( \eta ) = \xi$. Then $\eta = \eta_{1} \cdots \eta_{m} \xi$ has $r ( \eta_{1} ) = v$. Let $1 \le k \le m$ be the largest integer such that $r ( \eta_{k} ) = v$. Since $\sigma^{k - 1} ( \eta ) \ne \xi$ by the choice of $m$, $\eta_{k} \cdots \eta_{m} \ne \mu$. Further, $r ( \eta_{k} \cdots \eta_{m} ) = v = s ( \eta_{k} \cdots \eta_{m} )$. Hence $\eta_{k} \cdots \eta_{m}$ is a return path in $E$ based at $v$, distinct from $\mu$. Thus $E$ satisfies Condition (K).  
\end{example}

\section{Primitive ideals in crossed products for classical systems}\label{secprim}

Suppose $(T,\tau)$ is a classical system. A closed invariant subset $Y$ of $T$ is a {\em maximal head} if for every pair $y_1,y_2\in Y$ and neighbourhoods $V_1$ of $y_1$ and $V_2$ of $y_2$, there exist points $x_1\in V_1,x_2\in V_2$ and $m,n\in\N$ with $\tau^m(x_1)=\tau^n(x_2)$. 

We claim that if $t\in T$ is periodic, then $\overline{\tau^{-\N}(t)}$ is a maximal head. Since $\tau^{-\N}(t)$ is nonempty and invariant, $\overline{\tau^{-\N}(t)}$ is a closed nonempty invariant subset of $T$. Given $y_1, y_2\in \overline{\tau^{-\N}(t)}$ and neighbourhoods $V_1$ of $y_1$ and $V_2$ of $y_2$ we know $\tau^{-\N}(t)\cap V_1\not=\emptyset\not= \tau^{-\N}(t)\cap V_2$. So there exist $x_1\in\tau^{-\N}(t)\cap V_1$ and $x_2\in\tau^{-\N}(t)\cap V_2$, and there are $m,n\in\N$ with $\tau^m(x_1)=t=\tau^n(x_2)$.

If $t\in T$ is periodic with period $n$, then we call $\beta:=\{\tau^k(t):0\le k\le n\}$ a {\em cycle}. The cycle $\beta$ is {\em discrete} if $t$ is isolated in $\overline{\tau^{-\N}(\beta)}:=\overline{\tau^{-\N}(t)}$. Each $\tau^k(t)$ is then isolated, and so each $\delta_{\tau^k(t)}\in C_c(\,\overline{\tau^{-\N}(\beta)}\,)$.

\begin{thm}\label{main thm for prim ideals section}
Suppose $(T,\tau)$ is a classical system and $T$ is second-countable. 

\smallskip
\textnormal{(a)} Suppose $Y$ is a maximal head in $T$. Then the ideal $I_Y$ defined in Theorem~\textup{\ref{T:gaugeideal}} is primitive if and only if there is no discrete cycle $\beta$ with $Y=\overline{\tau^{-\N}(\beta)}$.

\smallskip
\textnormal{(b)} Suppose $\beta$ is a discrete cycle with $|\beta|=n$ and denote $Y:=\overline{\tau^{-\N}(\beta)}$. Choose $t\in\beta$, $f\in C_c(T)$ with $f|_Y=\delta_t$, and $g_i\in C_c(T)$ with $g_{i} |_{Y} = \sqrt{|\tau^{-1} (\tau^{i+1} (t))|} \delta_{\tau^{i} ( t )}$ for $0\le i\le n-1$. Then for each $w\in\T$ the ideal $I_{\beta,w}$ generated by
\[
\{k_{M_L}(g_0)\dots k_{M_L}(g_{n-1})-wk_A(f)\}\cup I_Y
\]
does not depend on the choice of $t\in\beta$ or functions $f,g_i$, and is primitive.

\smallskip
\textnormal{(c)} Every primitive ideal $I$ of $C_0(T){\rtimes}_{\alpha,L}\N$ has the form $I_Y$ for $Y$ given by \textup{(\ref{form for I Y})} or $I_{\beta,w}$ for a unique choice of cycle $\beta$ and $w\in\T$.

\smallskip
\textnormal{(d)} The ideals $I_Y$ are gauge-invariant, and the ideals $I_{\beta,w}$ are not.
\end{thm}

\begin{proof}
We first prove that $I_{\beta,w}$ does not depend on the choice of $f$. Write
\[
x:= k_{M_L}(g_0)\dots k_{M_L}(g_{n-1})-wk_A(f).
\]
Suppose $h\in C_c(T)$ satisfies $h|_Y=\delta_t$ and let $\tilde{x}=k_{M_L}(g_0)\dots k_{M_L}(g_{n-1})-wk_A(h)$. Then $h-f\in C_c(T\setminus Y)$, and $k_A(h-f)\in I_Y$. It follows that $\tilde{x}-x=wk_A(h-f)\in I_Y$, and so $\{x\}\cup I_Y$ and $\{\tilde{x}\}\cup I_Y$ generate the same ideal. 

To prove that $I_{\beta,w}$ does not depend on the choice of $g_i$ for any $0\le i\le n-1$ we do it for $g_0$. Recall from the proof of Theorem~\ref{T:gaugeideal} that $k_{M_L}(g)\in I_Y$ for all $g\in C_c(T\setminus Y)$.  Suppose $p_0\in C_c(T)$ satisfies $p_0|_Y=\delta_{t}$, and let $\tilde{x}=k_{M_L}(p_0)k_{M_L}(g_1)\dots k_{M_L}(g_{n-1})-wk_A(f)$. Then $p_0-g_0\in C_c(T\setminus Y)$, and $k_{M_L}(p_0-g_0)\in I_Y$. Thus $\tilde{x}-x\in I_Y$, and  $\{x\}\cup I_Y$ and $\{\tilde{x}\}\cup I_Y$ generate the same ideal.

To prove that $I_{\beta,w}$ does not depend on the choice of $t\in\beta$ it suffices to show that for $h\in C_c(T)$ with $h|_Y=\delta_{\tau(t)}$ and 
\[
\tilde{x}= k_{M_L}(g_1)\dots k_{M_L}(g_{n-1})k_{M_L}(g_0)-wk_A(h),
\]
the sets $\{x\}\cup I_Y$ and $\{\tilde{x}\}\cup I_Y$ generate the same ideal. We have 
\[
k_{M_L}(g_0)^*xk_{M_L}(g_0) = k_{M_L}({\langle g_0,g_0\rangle}_Lg_1)k_{M_L}(g_2)\dots k_{M_L}(g_{n-1})k_{M_L}(g_0)-wk_A({\langle g_0,fg_0\rangle}_L),
\]
and routine calculations show that
\[
{\langle g_0,g_0\rangle}_Lg_1|_Y=\sqrt{|\tau^{-1}(\tau^2(t))|}\delta_{\tau(t)}=g_1|_Y\quad\text{and}\quad{\langle g_0,fg_0\rangle}_L|_Y=\delta_{\tau(t)}=h|_Y.
\]
So $\{\tilde{x}\}\cup I_Y$ generates the same ideal as $\{k_{M_L}(g_0)^*xk_{M_L}(g_0)\}\cup I_Y$, which is contained in the ideal generated by $\{x\}\cup I_Y$.  

To get the reverse containment we assume without loss of generality that $\tau$ is injective on $\supp g_0$, and write $m:={(|\tau^{-1}(\tau(t))|)}^{-1}$, $d(s):=|\tau^{-1}(\tau(s))|$ for $s\in T$. We have
\begin{align}
k_{M_L}(g_0)&\tilde{x}k_{M_L}(mdg_0)^* =\nonumber\\
&k_{M_L}(g_0)\dots k_{M_L}(g_{n-1}){(k_{M_L},k_A)}^{(1)}(\Theta_{g_0,mdg_0})-w{(k_{M_L},k_A)}^{(1)}(\Theta_{g_0\alpha(h),mdg_0}).\label{eq with the 1 maps}
\end{align}
It follows from the injectivity of $\tau$ on $\supp g_0$ that
\[
\Theta_{g_0,mdg_0}=\phi(m{|g_0|}^2)\quad\text{and}\quad\Theta_{g_0\alpha(h),mdg_0}=\phi(m{|g_0|}^2\alpha(h)),
\]
and since $(k_{M_L},k_{A})$ is coisometric on $C_0(T)$, we have ${(k_{M_L},k_A)}^{(1)}(\Theta_{g_0,mdg_0})=k_A(m{|g_0|}^2)$ and ${(k_{M_L},k_A)}^{(1)}(\Theta_{g_0\alpha(h),mdg_0})=k_A(m{|g_0|}^2\alpha(h))$. The right-hand side of (\ref{eq with the 1 maps}) then becomes
\[
k_{M_L}(g_0)\dots k_{M_L}(g_{n-2})k_{M_L}(g_{n-1}\alpha(m{|g_0|}^2))-wk_A(m{|g_0|}^2\alpha(h)).
\]
Routine calculations show that
\[
\big(g_{n-1}\alpha(m{|g_0|}^2)\big)|_Y=\sqrt{|\tau^{-1}(t)|}\delta_{\tau^{n-1}(t)}=g_{n-1}|_Y\quad\text{and}\quad\big(m{|g_0|}^2\alpha(h)\big)|_Y=\delta_t=f|_Y,
\]
and so $\{x\}\cup I_Y$ generates the same ideal as $\{k_{M_L}(g_0)^*\tilde{x}k_{M_L}(mdg_0)\}\cup I_Y$, which is contained in the ideal generated by $\{\tilde{x}\}\cup I_Y$. Hence $\{x\}\cup I_Y$ and $\{\tilde{x}\}\cup I_Y$ generate the same ideal, and we have finished proving that $I_{\beta,w}$ does not depend on choices.
 
We now want to apply Theorem~11.14 and Corollary~12.3 of \cite{k2} to $E=(T,T,\tau,\id)$, so we again have to reconcile our definitions with Katsura's.

The sets in \cite[Definition~1.3]{k2} are $T_{\sce}=\emptyset$ and $T_{\fin}=T=T_{\rg}$, so $Y\subset T$ is invariant if and only if it is invariant in the sense of \cite[Definition~2.1]{k2}. We have already seen that $t\in T$ is periodic and isolated in $\tau^{-\N}(t)$ if and only if $t$ is an element of $\Per(E)$ given in \cite[Definition~7.1]{k2}. Thus $Y\subset T$ is a maximal head if and only if it is a maximal head as in \cite[Definition~4.12]{k2}. The definition of $\MM_{\Per}(E)$  in the middle of \cite[page~1839]{k2} shows that $\{\overline{\tau^{-\N}(\beta)}:\beta\text{ a discrete cycle}\}=\MM_{\Per}(E)$.

We claim that for $Y$ a maximal head as in (a) we have $I_Y=P_Y$, where $P_Y$ is given in \cite[Definition~11.4]{k2}. We have already seen in the proof of Theorem~\ref{T:gaugeideal} that $I_Y=I_{\rho}$, where $\rho$ is the admissible pair $(Y,\emptyset)$ and $I_{\rho}$ is given in \cite[Definition~3.3]{k2}. The ideal $P_Y$ is defined to be $I_{\rho}$ for such $Y$, so the claim follows. 

We now claim that for $w\in\T$, $\beta$ a discrete cycle and $Y:=\overline{\tau^{-\N}(\beta)}$ we have $I_{\beta,w}=P_{Y,w}$, where $P_{Y,w}$ is given in \cite[Definition~11.8]{k2}. Write $\beta=\{\tau^k(t):0\le k\le n-1\}$, choose $f\in C_c(T)$ such that $f|_Y=\delta_t$, and for $0\le i\le n-1$ choose functions $g_i\in C_c(T)$ with $g_i|_Y=\delta_{\tau^i(t)}$. The ideal $P_{Y,w}$ is generated by
\[
\{ k_{C_{\tau} (T)} ( g_{0} ) k_{C_{\tau} (T)} ( g_{1} ) \cdots k_{C_{\tau} (T)} ( g_{n-1} ) - w k_{A} ( f ) \} \cup I_{Y}.
\]
The isomorphism $U \colon C_{\tau}(T) \to M_{L}$ of \S\ref{The system (C_0(X),alpha,L)} satisfies $k_{C_{\tau} (T)} := k_{M_{L}} \circ U$, so $P_{Y,w}$ is generated by
\[
\{ k_{M_{L}} ( U(g_{0}) ) k_{M_{L}} ( U(g_{1}) ) \cdots k_{M_{L}} ( U(g_{n-1}) ) - w k_{A} ( f ) \} \cup I_{Y}.
\]
For $0 \le i \le n-1$ we have $U(g_{i}) |_{Y} = \sqrt{|\tau^{-1} (\tau^{i+1} (t))|} \delta_{\tau^{i} ( t )}$, so $P_{Y,w} = I_{\beta,w}$.

The set $\BV(E)$ given in \cite[page~1837]{k2} is empty, so the result now follows from \cite[Theorem~11.14]{k2} and \cite[Corollary~12.3]{k2} (which needs second-countability). 
\end{proof}

\subsection{The primitive ideals of graph algebras}\label{SS:PrimGA}

Let $E$ be a locally finite graph with no sources. As in \cite{hs}, a \emph{maximal head} is a non-empty subset $M$ of $E^{0}$ such that
\begin{enumerate}\item[]
\begin{enumerate}\item[(MH1)] if $v \in E^{0}$, $w \in M$, and $v \le w$ then $v \in M$;

\smallskip

\item[(MH2)] if $v \in M$, then there exists $e \in E^{1}$ with $r ( e ) = v$ and $s ( e ) \in M$; and

\smallskip

\item[(MH3)] for every $v, w \in M$ there exists $y \in M$ such that $v \le y$ and $w \le y$.
\end{enumerate}\end{enumerate}
We write $\MM ( E )$ for the set of maximal heads in $E$, and $\MM_{l} ( E )$ for the set of maximal heads $M$ containing a return path without an entry in $M$. Lemma~2.1 of \cite{hs} says that $M \in \MM_{l} ( E )$ if and only if there is a cycle in $M$ without an entry in $M$. 

The following result was proved for arbitrary directed graphs in \cite[Corollary 2.12]{hs}. 

\begin{thm}\label{main res for prim ideal of graph alg}
Suppose $E$ is a locally finite directed graph with no sources, and denote by $\{s,p\}$ the universal Cuntz-Krieger $E$-family in $C^*(E)$.   

\smallskip
\textnormal{(a)} Suppose $M\subset E^0$ is a maximal head. Then the ideal $I_{E^0\setminus M}$ in $C^*(E)$ generated by $\{p_v:v\in E^0\setminus M\}$ is primitive if and only if every cycle in $M$ has an entry.

\smallskip
\textnormal{(b)} Suppose $M\subset E^0$ is a maximal head and let $\mu_1\dots\mu_n$ be a cycle in $M$ without an entry in $M$. Then for each $w\in\T$ the ideal $I_{M,w}$ generated by
\[
\{s_{\mu_1}\dots s_{\mu_n}-wp_{r(\mu_1)}\}\cup I_{E^0\setminus M}
\]
does not depend on the choice of cycle $\mu_1\dots\mu_n$, and is primitive.

\smallskip
\textnormal{(c)} Every primitive ideal $I$ of $C^*(E)$ is $I_{E^0\setminus M}$ for $M=\{v\in E^0:p_v\in I\}$ or $I_{M,w}$ for a unique $w\in\T$ and a unique  maximal head $M$ containing a cycle without an entry.

\smallskip
\textnormal{(d)} The ideals $I_{E^0\setminus M}$ are gauge-invariant, and the ideals $I_{M,w}$ are not.
\end{thm}

\begin{remark}\label{stuff about the sets}
We claim that $Y\mapsto H_Y:=\{r(\xi):\xi\in E^\infty\setminus Y\}$ is a bijection from the set of closed invariant subsets of $E^{\infty}$ onto the set of saturated and hereditary subsets of $E^0$, with inverse $H\mapsto Y_H:=\{\xi : r(\xi)\in E^0\setminus H\}$. 

Suppose $\xi\in Y_{H_{Y}}$. Since $H$ is hereditary, $r ( \xi_{i} ) \notin H_{Y}$ for all $i \ge 1$, and so $s ( \xi_{i} ) \notin H_{Y}$ for all $i\ge 1$. For each $i\ge 1$ there exists $\eta^{i} \in Y$ such that $r ( \eta^{i} ) = s ( \xi_{i} )$, and since $Y$ is invariant, $\xi_{1} \xi_{2} \cdots \xi_{i} \eta^{i} \in Y$. The sequence $( \xi_{1} \xi_{2} \cdots \xi_{i} \eta^{i} )_{i = 1}^{\infty}$ converges in $E^{\infty}$ to $\xi$. Since $Y$ is closed, $\xi \in Y$ and hence $Y_{H_Y}\subset Y$. Conversely, we have
\[
\lambda\in Y\Longrightarrow r(\lambda)\not\in H_Y\Longrightarrow \lambda\in Y_{H_Y}.
\]
So $Y\subset Y_{H_Y}$, and hence $Y_{H_Y}=Y$.

Fix a saturated hereditary subset $H$ of $E^{0}$ and suppose $v \notin H$. Let $\xi\in Y_H$ with $r(\xi)=v$. Then $v\not\in H_{Y_H}$, and so  $H_{Y_H}\subset H$. Conversely, let $v \in H$. If $\xi \in Y_{H}$, then $r ( \xi ) \neq v$, so $v \in H_{Y_{H}}$. Hence $H\subset H_{Y_H}$, and so $H_{Y_H}=H$.
\end{remark}

\begin{lemma}\label{max heads match}
Let $E$ be a locally finite directed graph with no sources. The map $M \mapsto Y_{E^{0} \setminus M}$ is a bijection from $\MM(E)$ onto the set of maximal heads in $E^{\infty}$, with inverse $Y \mapsto E^{0} \setminus H_{Y}$, and it maps $\MM_{l} ( E )$ onto $\{\overline{\sigma^{-\N}(\beta)}:\beta\textup{ a discrete cycle in }E^{\infty}\}$.
\end{lemma}

\begin{proof}
Let $M\in\MM(E)$. Since $M$ is hereditary and saturated, $Y_{E^0\setminus M}$ is closed and invariant. Suppose $\xi^1,\xi^2\in Y_{E^0\setminus M}$ and consider the neighbourhoods $Z(\xi_1^1\dots\xi_m^1)$ of $\xi^1$ and $Z(\xi_1^2\dots\xi_n^2)$ of $\xi^2$. It follows from (MH3) that there exists $v\in M$ and paths $\lambda,\mu$ with $s(\lambda)=v=s(\mu)$, $r(\lambda)=s(\xi_m^1)$ and $r(\mu)=s(\xi_n^2)$. Take $\eta\in E^{\infty}$ with $r(\eta)=v$, and let $\eta^1:=\xi_1^1\dots\xi_m^1\lambda\eta$ and $\eta^2:=\xi_1^2\dots\xi_n^2\mu\eta$. Then we have $\eta_1\in Z(\xi_1^1\dots\xi_m^1)$, $\eta^2\in Z(\xi_1^2\dots\xi_n^2)$ and $\sigma^{m+|\lambda|}(\eta^1)=\eta=\sigma^{n+|\mu|}(\eta^2)$. So $Y_{E^0\setminus M}$ is a maximal head in $E^{\infty}$. 

Let $Y$ be a maximal head in $E^{\infty}$. To see that $E^0\setminus H_Y=\{r(\xi):\xi\in Y\}$ satisfies (MH1), let $v\in E^0$ and $r(\xi)\in E^0\setminus H_Y$ with $v\le r(\xi)$. Then there exists a path $\mu$ with $s(\mu)=r(\xi)$ and $r(\mu)=v$. Since $\mu\xi\in\sigma^{-|\mu|}(\xi)$, it follows from the invariance of $Y$ that $\mu\xi\in Y$. Hence $v=r(\mu\xi)\in E^0\setminus H_Y$.

It follows from the invariance of $Y$ that for $r(\xi)\in E^0\setminus H_Y$ we have $\sigma(\xi)\in Y$, and so $r(\sigma(\xi))\in E^0\setminus H_Y$. Then $\xi_1\in E^1$ satisfies $r(\xi_1)=r(\xi)\in E^0\setminus H_Y$ and $s(\xi_1)=r(\sigma(\xi))\in E^0\setminus H_Y$, and so $E^0\setminus H_Y$ satisfies (MH2).

Since $Y$ is a maximal head, for each $r(\xi),r(\eta)\in E^0\setminus H_Y$ there exists $\xi',\eta'\in E^{\infty}$ with $r(\xi')=r(\xi)$ and $r(\eta')=r(\eta)$, and $m,n\in\N$ with $\sigma^m(\xi')=\sigma^n(\eta')$. Since $Y$ is invariant, we have $\sigma^m(\xi')\in Y$, so $r(\sigma^m(\xi'))\in E^0\setminus H_Y$ and satisfies $r(\xi),r(\eta)\le r(\sigma^m(\xi'))$. So (MH3) is satisfied. The first assertion in the result now follows from Remark~\ref{stuff about the sets}.

Now suppose $M\in\MM_l(E)$ and $\mu=\mu_1\dots\mu_n$ is a cycle in $M$ without an entry in $M$. We claim that for $\eta:=\mu\mu\dots\in E^{\infty}$ the set $\beta:=\{\sigma^k(\eta):0\le k\le n-1\}$ is a discrete cycle with $Y_{E^0\setminus M}=\overline{\sigma^{-\N}(\beta)}$. To see that $\eta$ is isolated in $\overline{\sigma^{-\N}(\beta)}$ suppose that $\xi\in\sigma^{-\N}(\beta) \cap Z ( r ( \eta ) )$. Then $\sigma^{m} ( \xi ) = \eta$ for some $m \in \N$, and $r ( \xi ) = r ( \eta )$. So $\xi = \xi_{1} \cdots \xi_{m} \eta$ where $r ( \xi_{1} ) = r ( \eta )$. Since $s ( \xi_{m} ) = r ( \eta ) \in M$, it follows from (MH1) that $r ( \xi_{i} ) \in M$ for each $1 \le i \le m$. If $\xi_{1} \ne \eta_{1} = \mu_{1}$, then $\xi_{1}$ is an entry for $\mu$ in $M$, so we must have $\xi_{1} = \mu_{1}$. Continuing in this manner for $2 \le i \le m$ gives $\xi = \eta$. So $\sigma^{-\N}(\beta) \cap Z ( r ( \eta ) ) = \{ \eta \}$, and hence $\eta$ is isolated in $\overline{\sigma^{-\N}(\beta)}$. 

Since $\eta\in Y_{E^0\setminus M}$ and $Y_{E^0\setminus M}$ is invariant, we have $\sigma^{-\N}(\beta)\subset Y_{E^0\setminus M}$. Since $Y_{E^0\setminus M}$ is closed, we have $\overline{\sigma^{-\N}(\beta)}\subset Y_{E^0\setminus M}$. For the reverse containment, let $\xi\in Y_{E^0\setminus M}$. If $\xi\in \sigma^{-\N}(\beta)$, then $\xi\in \overline{\sigma^{-\N}(\beta)}$, so we assume $\xi\not\in \sigma^{-\N}(\beta)$. It suffices to show that $\sigma^{-\N}(\beta)\cap Z(\xi_1\dots\xi_j)\not=\emptyset$ for all $j$. Consider the points and neighbourhoods $\eta\in Z(r(\eta))$ and $\xi\in Z(\xi_1\dots\xi_j)$. Since $Y_{E^0\setminus M}$ is a maximal head, there exists $\lambda^1\in Z(r(\eta))$, $\lambda^2\in Z(\xi_1\dots\xi_j)$ and $m,n\in\N$ with $\sigma^m(\lambda^1)=\sigma^n(\lambda^2)$. Since $Y_{E^0\setminus M}$ is invariant, we have $\lambda^1\in Y_{E^0\setminus M}$, and so $r(\lambda_i^1)\in M$ for all $i$. Since $r(\lambda^1)=r(\eta)$ and $\mu$ does not have an entry, we must have $\lambda^1=\eta$. So $\sigma^n(\lambda^2)=\sigma^m(\lambda^1)\in \sigma^{-\N}(\beta)$, which implies $\lambda^2\in \sigma^{-\N}(\beta)$. So $\lambda^2\in \sigma^{-\N}(\beta)\cap Z(\xi_1\dots\xi_j)$.               

To see that $M \mapsto Y_{E^{0} \setminus M}$ maps $\MM_l(E)$ onto  $\{\overline{\sigma^{-\N}(\beta)}:\beta\text{ a discrete cycle in }E^{\infty}\}$, we suppose $\beta=\{\sigma^k(\xi):0\le k\le n-1\}$ is a discrete cycle, and let $Y=\overline{\sigma^{-\N}(\beta)}$. The bijection sends $E^0\setminus H_Y$ to $Y$, so we need to show that $E^0\setminus H_Y\in\MM_l(E)$. We know that $\xi_1\dots\xi_n$ is a return path in $E$. Since $\xi\in Y$, $r(\xi_1)\in E^0\setminus H_Y$. Then (MH1) implies that $r(\xi_i)\in E^0\setminus H_Y$ for $1\le i\le n$. We suppose that $\xi_1\dots\xi_n$ has an entry in $E^{0} \setminus H_{Y}$, and look for a contradiction. There exist $e \in E^{1}$ and $1 \le j \le n$ such that $e \ne \xi_{j}$, $r ( e ) = r ( \xi_{j} )$, and $s ( e ) \in E^{0} \setminus H_{Y}$. Since $s ( e ) \notin H_{Y}$, there exists $\eta \in Y$ such that $r ( \eta ) = s ( e )$. Choose $m \ge k+1$ such that $\xi_{m} = \xi_{j}$, and consider the infinite path $\xi_{1} \cdots \xi_{m-1} e \eta$. Since $Y$ is invariant and $\eta \in Y$,  $\xi_{1} \cdots \xi_{m-1} e \eta \in Z ( \xi_{1} \cdots \xi_{k} ) \cap Y$. Moreover, since $e \ne \xi_{m}$, $\xi_{1} \cdots \xi_{m-1} e \eta \ne \xi$. Thus $\xi$ is not isolated in $Y$, which is a contradiction. Therefore, the return path $\xi_{1} \cdots \xi_{n}$ must have no entries in $E^{0} \setminus H_{Y}$, and hence $E^{0} \setminus H_{Y} \in \MM_{l} ( E )$.    
\end{proof}

\begin{proof}[Proof of Theorem~\ref{main res for prim ideal of graph alg}]
Let $w\in \T$, $M\in\MM_l(E)$ and $\mu_1\dots\mu_n\in M$ be a cycle without an entry in $M$, and note that all such cycles are cyclic permutations of each other. The Cuntz-Krieger relations imply that for $q_{i} := s_{\mu_{1}} s_{\mu_{2}} \cdots s_{\mu_{i-1}}$ and $r_{i} := s_{\mu_{i}} s_{\mu_{i+1}} \cdots s_{\mu_{n}}$ we have 
\[
s_{\mu_{i}} s_{\mu_{i+1}} \cdots s_{\mu_{n}} s_{\mu_{1}} s_{\mu_{2}} \cdots s_{\mu_{i-1}} - w p_{r ( \mu_{i} )}
= q_{i}^{*} ( s_{\mu_{1}} s_{\mu_{2}} \cdots s_{\mu_{n}} - w p_{r ( \mu_{1} )} ) q_{i}
\]
and
\[
s_{\mu_{1}} s_{\mu_{2}} \cdots s_{\mu_{n}} - w p_{r ( \mu_{1} )}
= r_{i}^{*} ( s_{\mu_{i}} s_{\mu_{i+1}} \cdots s_{\mu_{n}} s_{\mu_{1}} s_{\mu_{2}} \cdots s_{\mu_{i-1}} - w p_{r ( \mu_{i} )} ) r_{i}.
\]
Thus the ideal $I_{M,w}$ does not depend on the choice of the cycle $\mu$.

Recall from Theorem~\ref{T:GAasECP} that for $\{S_e,P_v\}$ given by (\ref{forms for ck family}) there exists an isomorphism $\pi_{S,P}:C^*(E)\to C_0(E^{\infty}){\rtimes}_{\alpha,L}\N$ satisfying $\pi_{S,P}(s_e)=S_e$ for each $e\in E^1$, and $\pi_{S,P}(p_v)=P_v$ for each $v\in E^0$. We can apply the arguments in Example ~\ref{graph ex for gii} to see that for each $M\in\MM(E)$ in which every cycle has an entry we have $\pi_{S,P}(I_{E^0\setminus M})=I_{Y_{E^0\setminus M}}$.

Now suppose $M\in \MM_l(E)$ and $\mu=\mu_1\dots\mu_n\in M$ is a cycle without an entry in $M$. We saw in the proof of Lemma~\ref{max heads match} that for $\eta:=\mu\mu\dots\in E^{\infty}$ the set $\beta=\{\sigma^k(\eta):0\le k\le n-1\}$ is a discrete cycle with $Y_{E^0\setminus M}=\overline{\sigma^{-\N}(\beta)}$. We claim that for $0 \le i \le n-1$ the function $g_i\in C_c(E^{\infty})$ given by $g_i(\xi)=\sqrt{|\tau^{-1}(\tau^{i+1}(\xi))|}\chi_{Z ( \mu_{i} )}$ has restriction $\sqrt{|\tau^{-1}(\tau^{i+1}(\xi))|}\delta_{\sigma^{i-1} ( \eta )}$ on $Y_{E^{0} \setminus M}$. Clearly $\chi_{Z ( \mu_{i} )} |_{Y_{E^{0} \setminus M}} ( \sigma^{i-1} ( \eta ) ) = 1$. Suppose that $\xi = \xi_{1} \xi_{2} \cdots \in Y_{E^{0} \setminus M} \cap Z ( \mu_{i} )$. Since $\xi \in Y_{E^{0} \setminus M}$, $r ( \xi_{j} ) \in M$ for $j \ge 1$. If we have $\xi \ne \sigma^{i-1} ( \eta ) = \eta_{i} \eta_{i+1} \cdots$, then noting that $\xi_{1} = \mu_{i} = \eta_{i}$, we can choose the smallest $m \ge 2$ such that $\xi_{m} \ne \eta_{i+m-1}$. Since $s ( \xi_{m} ) \in M$, $\xi_{m}$ is an entry for $\mu$ in $M$, which is a contradiction. So $\xi = \sigma^{i-1} ( \eta )$. Therefore $\chi_{Z ( \mu_{i} )} |_{Y_{E^{0} \setminus M}} = \delta_{\sigma^{i-1} ( \eta )}$, and the claim follows. A similar argument shows that $\chi_{Z ( r ( \mu_{1} ) )} |_{Y_{E^{0} \setminus M}}$ is the characteristic function $\delta_{\eta}$. 

The ideal $I_{\beta,w}$ is generated by the set 
\[
\left\{k_{M_L} ( g_0 )\dots k_{M_L} ( g_{n-1} ) - w k_{A} ( \chi_{Z ( r ( \mu_{1} ) )} )\right\}\cup I_{Y_{E^{0} \setminus M}}.
\]
Since $\sqrt{c(s(\mu_i))}=\sqrt{|\tau^{-1}(\tau^{i}(\xi))|}$, $I_{\beta,w}$ is also generated by
\[
\left\{\pi_{S,P}\left(s_{\mu_1}\dots s_{\mu_n}-wp_{r(\mu_1)}\right)\right\}\cup\pi_{S,P}(I_{E^0\setminus M}),
\]
which is $\pi_{S,P}(I_{E^0\setminus M,w})$.

The result now follows by applying Theorem~\ref{main thm for prim ideals section} to the system $(E^{\infty},\sigma)$.
\end{proof}

\section{Conclusions}\label{Conclusions}

In extending Exel's theory to non-unital algebras, we have had to make choices. We have already mentioned one such issue in Remark~\ref{normalising factor}: even for a classical system $(T,\tau)$ there are different choices of transfer operator. We have mainly used the normalised version which is defined on all of $C_0(T)$. However, when we used the isomorphism with the topological-graph algebra $\OO(E)$, we were effectively switching to the unnormalised version, which is only densely defined on $C_0(T)$. We chose not to try to develop a general theory for systems with densely-defined transfer operators, though we think the topic is potentially interesting, and this is one possible direction for further work. Here we discuss several other possible directions.

To get a bounded transfer operator, we had to restrict attention to locally finite graphs. To get a theory which applies to arbitrary graphs, we would need to use the boundary $\partial E$, which is formed by adding to $E^\infty$ the paths which start at a source or a vertex $v$ where $r^{-1}(v)$ is infinite. Then the shift is not everywhere defined, so we need to allow partially defined maps $\tau$, as is done for the compact case in \cite{er}. One could then directly define a topological graph (that is, with no normalising factor), so that Katsura's theory applies, and view his algebra as the crossed product. Such methods, though, could only be used for classical systems.

A second possibility which appeals to us is guided by what might work for actions of semigroups. From this point of view, it seems best to drop the normalising factor: the square $L^2$ of the normalised transfer operator $L$ for a classical system $(T,\tau)$ need not be the normalised transfer operator for $\alpha^2$ (as examples from graphs show). So we come back to densely-defined transfer operators. However, rather than work out some axioms, we think it might be best to concentrate on the modules $M_L$, which can be built by completing a dense subspace such as $C_c(T)$, work out conditions under which these modules form a product system over the semigroup in the sense of Fowler \cite{f}, and define the Exel crossed product to be the Cuntz-Pimsner algebra of the product system. A start on such a theory has been made by Larsen \cite{l}, though she deals only with bounded transfer operators. One problem with such an approach is that there is not yet a generally accepted notion of Cuntz-Pimsner algebra for product systems (see the discussion at the start of \cite{sy}). Nevertheless, examples and intuition from Exel systems might be a fertile source of interesting product systems, and a useful contribution to the general theory.


\begin{thebibliography}{20}

\bibitem{a} S. Adji,  {\em Invariant ideals of crossed products by semigroups of endomorphisms}, Functional Analysis and Global Analysis (T. Sunada and P. W. Sy, Eds.), Springer-Verlag, Singapore, 1997, pages~1--8.



\bibitem{alnr} S. Adji, M. Laca, M. Nilsen and I. Raeburn, {\em Crossed products by semigroups of endomorphisms and the Toeplitz algebras of ordered groups}, Proc. Amer. Math. Soc. {\bf 122} (1994), 1133--1141.

\bibitem{bhrs} T. Bates, J. H. Hong, I. Raeburn and W. Szyma\'{n}ski, {\em The ideal structure of the $C^{*}$-algebras of infinite graphs}, Illinois J. Math. {\bf 46} (2002), 1159--1176.

\bibitem{bprs} T. Bates, D. Pask, I. Raeburn and W. Szyma\'{n}ski, {\em The $C^*$-algebras of row-finite graphs}, New York J. Math. {\bf 6} (2000), 307--324.


\bibitem{br} N. Brownlowe and I. Raeburn, {\emÊ Exel's crossed product and relative Cuntz-Pimsner algebras}, Math. Proc. Camb. Phil. Soc. {\bf 141} (2006), 497--508.

\bibitem{c} J. Cuntz, {\em The internal structure of simple $C^*$-algebras}, Proc. Sympos. Pure Math., vol. 38, Amer. Math. Soc., Providence, 1982, pages 85--115.

\bibitem{e1} R. Exel, {\emÊ A new look at the crossed-product of a $C^*$-algebra by an endomorphism}, Ergodic Theory Dynam. Systems {\bf 23} (2003), 1--18.
Ê 
\bibitem{e2} R. Exel, {\emÊ Crossed-products by finite index endomorphisms and KMS states}, J. Funct. Anal. {\bf 199} (2003), 153--188.

\bibitem{er} R. Exel and D. Royer, {\em The crossed product by a partial endomorphism},  Bull. Braz. Math. Soc. {\bf 38} (2007), 219--261.

\bibitem{ev} R. Exel and A. Vershik, {\emÊ $C^*$-algebras of irreversible dynamical systems}, Canad. J. Math. {\bf 58} (2006), 39--63.

\bibitem{f} N. J. Fowler, {\em Discrete product systems of Hilbert bimodules}, Pacific J. Math. {\bf 204} (2002), 335--375.


\bibitem{fmr} N. J. Fowler, P. S. Muhly and I. Raeburn, {\em Representations of Cuntz-Pimsner algebras}, Indiana Univ. Math. J. {\bf 52} (2003), 569--605.

\bibitem{fr} N. J. Fowler and I. Raeburn, {\em The Toeplitz algebra of a Hilbert bimodule}, Indiana Univ. Math. J. {\bf 48} (1999), 155--181.

\bibitem{hs} J. H. Hong and W. Szyma\'{n}ski, {\em The primitive ideal space of the $C^{*}$-algebras of infinite graphs}, J. Math. Soc. Japan {\bf 56} (2004), 45--64.


\bibitem{im} M. Ionescu and P. S. Muhly, \emph{Groupoid methods in wavelet analysis}, Group Representations, Ergodic Theory, and Mathematical Physics: A Tribute to George W. Mackey, Contemp. Math., vol. 449, Amer. Math. Soc., Providence, 2008, pages  193--208.

\bibitem{kat} T. Katsura, {\em On $C\sp *$-algebras associated with $C\sp *$-correspondences},  J. Funct. Anal.  {\bf 217}  (2004), 366--401.

\bibitem{k} T. Katsura, {\em A class of $C^*$-algebras generalizing both graph algebras and homeomorphism $C^*$-algebras \textup{I}, Fundamental results}, Trans. Amer. Math. Soc. {\bf 356} (2004), 4287--4322.

\bibitem{k2} T. Katsura, {\em A class of $C^{*}$-algebras generalizing both graph algebras and homeomorphism $C^{*}$-algebras \textup{III}, ideal structures}, Ergodic Theory Dynam. Systems {\bf 26} (2006), 1805--1854.

\bibitem{lr2} M. Laca and I. Raeburn, {\em Semigroup crossed products and the Toeplitz algebras of nonabelian groups}, J. Funct. Anal. {\bf 139} (1996), 415--440.

\bibitem{lr} M. Laca and I. Raeburn, {\em A semigroup crossed product arising
in number theory}, J. London Math. Soc. {\bf 59} (1999), 330--344.

\bibitem{l} N. S. Larsen, {\em Exel crossed products over abelian semigroups}, Ergodic Theory Dynam. Systems, to appear.

\bibitem{ms} P. S. Muhly and B. Solel, {\em Tensor algebras over $C^*$-correspondences (representations, dilations, and $C^*$-envelopes)\/}, J. Funct. Anal. {\bf 158} (1998), 389--457.


\bibitem{munk} J. R. Munkres, Topology, Second Edition, Prentice Hall, New Jersey, 2000.

\bibitem{mold} G. J. Murphy, {\em Ordered groups and Toeplitz algebras}, J. Operator Theory {\bf 18} (1987), 303--326.

\bibitem{m} G. J. Murphy, {\em Crossed products of $C^*$-algebras by endomorphisms}, Integral Equations Operator Theory {\bf 24} (1996), 298--319.

\bibitem{n} A. Nica, {\em $C^*$-algebras generated by isometries and
Wiener-Hopf operators}, J. Operator Theory {\bf 27} (1992), 17--52.

\bibitem{pas} W. Paschke, {\em The crossed product by an endomorphism},
  Proc. Amer. Math. Soc. {\bf 80} (1980), 113-118.

\bibitem{p} M. V. Pimsner, {\em A class of $C^*$-algebras generalizing both Cuntz-Krieger algebras and crossed products by $\Z$}, Fields Institute Commun. {\bf 12} (1997), 189--212.


\bibitem{r} I. Raeburn, Graph Algebras, CBMS Regional Conference Series in Math., vol. 103, Amer. Math. Soc., Providence, 2005.

\bibitem{rw} I. Raeburn and D. P. Williams, Morita Equivalence and Continuous-Trace $C^*$-Algebras, Math. Surveys and Monographs, vol. 60, Amer. Math. Soc., Providence, 1998.

\bibitem{sy} A. Sims and T. Yeend, {\em $C^*$-algebras associated to product systems of Hilbert bimodules}, J. Operator Theory, to appear; arXiv.math.OA/07123073.

\bibitem{s} P. J. Stacey, {\em Crossed products of $C^*$-algebras by $*$-endomorphisms}, J. Austral. Math. Soc. Ser. A {\bf 54} (1993), 204--212.

\end{thebibliography}
\end{document}